\documentclass{amsart}

\usepackage{amsmath, amssymb, amsthm, setspace, tikz-cd, hyperref}
\usepackage[shortcuts]{extdash}
\theoremstyle{plain}
\newtheorem{theorem}{Theorem}[section]
\newtheorem*{theoremA}{Theorem A}
\newtheorem*{theoremB}{Theorem B}
\newtheorem*{theoremC}{Theorem C}
\newtheorem{lemma}[theorem]{Lemma}
\newtheorem{proposition}[theorem]{Proposition}
\newtheorem{corollary}[theorem]{Corollary}

\theoremstyle{definition}
\newtheorem{definition}[theorem]{Definition}
\newtheorem{exampletemp}[theorem]{Example}
\newtheorem{remark}[theorem]{Remark}

\newtheorem*{notation}{Notation}

\newenvironment{example}{\begin{exampletemp}}{\hfill\qed\end{exampletemp}}
\begin{document}

\title[Cuntz-Pimsner Algebras, Crossed Products, and $K$\=/Theory]{Cuntz-Pimsner Algebras, Crossed Products, \\ and $K$\=/Theory}
\author{Christopher P. Schafhauser}
\address{Department of Mathematics \\ University of Nebraska -- Lincoln \\ 203 Avery Hall \\ Lincoln, NE 68588--0130}
\email{cschafhauser2@math.unl.edu}
\date{\today}

\begin{abstract}
  Suppose $A$ is a $C^*$-algebra and $H$ is a $C^*$\=/correspondence over $A$.  If $H$ is regular in the sense that the left action of $A$ is faithful and is given by compact operators, then we compute the $K$\=/theory of $\mathcal{O}_A(H) \rtimes \mathbb{T}$ where the action is the usual gauge action.  The case where $A$ is an AF-algebra is carefully analyzed.  In particular, if $A$ is AF, we show $\mathcal{O}_A(H) \rtimes \mathbb{T}$ is AF.  Combining this with Takai duality and an AF-embedding theorem of N. Brown, we show the conditions AF-embeddability, quasidiagonality, and stable finiteness are equivalent for $\mathcal{O}_A(H)$.  If $H$ is also assumed to be regular, these finiteness conditions can be characterized in terms of the ordered $K$\=/theory of $A$.
\end{abstract}

\maketitle

\section{Introduction}

A $C^*$\=/correspondence consists of $C^*$\=/algebras $A$ and $B$ and an $A$--$B$ bimodule $H$ together with a complete $B$--valued inner product which satisfies certain conditions.  In this case, $H$ can be thought of as a multi\=/valued, partially defined morphism from $A$ to $B$.  In fact if $H$ is ``row\=/finite'' then $H$ induces a positive group homomorphism $[H] : K_0(A) \rightarrow K_0(B)$ (see Section \ref{sec:KTheory}).  In the special case when $A = B$, Pimsner introduced in \cite{PimsnerCorr} a certain $C^*$\=/algebra $\mathcal{O}_A(H)$, now called the Cuntz--Pimsner algebra, which can be thought of as the crossed product of $A$ by the generalized morphism $H$ (see also \cite{KatsuraCorr1, KatsuraCorr2} for example).

The class of Cuntz--Pimsner algebras includes many naturally occurring $C^*$\=/algebras such as crossed products by $\mathbb{Z}$, crossed products by partial automorphisms, graph algebras, and much more (see \cite[Section 3]{KatsuraCorr1} for other examples).  Much of the structure of $\mathcal{O}_A(H)$ can be recovered from the underlying correspondence $H$ and the algebra $A$.  For instance, the ideal structure of $\mathcal{O}_A(H)$ is extensively studied in \cite{KatsuraCorr3}, and there is a Pimsner--Voiculescu six--term exact sequence in $K$\=/theory for $\mathcal{O}_A(H)$ given in \cite[Theorem 8.6]{KatsuraCorr2} (originally shown in \cite[Theorem 4.9]{PimsnerCorr} with mild hypotheses).

There is a natural gauge action $\gamma$ of $\mathbb{T}$ on  $\mathcal{O}_A(H)$.  By Takai duality, $\mathcal{O}_A(H)$ is Morita equivalent to $(\mathcal{O}_A(H) \rtimes_\gamma \mathbb{T}) \rtimes_{\hat{\gamma}} \mathbb{Z}$.  Thus understanding the structure of $\mathcal{O}_A(H) \rtimes_\gamma \mathbb{T}$ can give some insight into the structure of the algebra $\mathcal{O}_A(H)$. In this note, we calculate the $K$\=/theory of $\mathcal{O}_A(H) \rtimes_\gamma \mathbb{T}$.  In particular, we have the following result (see Section \ref{sec:SkewProduct} for a proof).

\begin{theoremA}
Suppose $A$ is a $C^*$-algebra and $H$ is a row\=/finite, faithful $C^*$\=/correspondence over $A$, then
\[ K_*(\mathcal{O}_A(H) \rtimes_{\gamma} \mathbb{T}) \cong \underset{\longrightarrow}{\lim} \, (K_*(A), [H]). \]
Moreover, the isomorphism preserves the order structure on $K_0$ and intertwines the automorphism $[\hat{\gamma}]$ on $K_*(\mathcal{O}_A(H) \rtimes_{\gamma} \mathbb{T})$ and the automorphism of $\underset{\longrightarrow}{\lim} \, (K_*(A), [H])$ induced by $[H]$.
\end{theoremA}

The proof involves a certain skew product construction similar to the graph skew product construction developed by Kumjian and Pask in \cite{KumjianPask} for graph $C^*$\=/algebras.  In particular, we build correspondences $H^n$ over $A^n$ for $n \in \mathbb{Z} \cup \{\infty\}$ such that
\[ \mathcal{O}_A(H) \rtimes_\gamma \mathbb{T} \cong \mathcal{O}_{A^\infty}(H^\infty) \cong \underset{\longrightarrow}{\lim} \, \mathcal{O}_{A^n}(H^n). \]
The $K$\=/theory of $\mathcal{O}_A(H) \rtimes_\gamma \mathbb{T}$ is calculated by showing that each $\mathcal{O}_{A^n}(H^n)$, is Morita equivalent to $A$.  The result then follows by the continuity and stability of $K$\=/theory.

The case where $A$ is an AF\=/algebra is also examined.  In particular, we have the following result (see Section \ref{sec:AFAlgebras}).

\begin{theoremB}
If $A$ is an AF\=/algebra and $H$ is a separable $C^*$\=/correspondence over $A$, then $\mathcal{O}_A(H) \rtimes_\gamma \mathbb{T}$ is AF.
\end{theoremB}

We end this note with some applications to AF\=/embeddability.  In \cite{BrownAFE}, N.~Brown characterized the AF-embeddability of crossed products of AF-algebras by the integers (see Theorem \ref{thm:BrownsEmbedding} below).  In Section \ref{sec:BrownsTheorem}, we prove the following generalization of Brown's Embedding Theorem.

\begin{theoremC}
Suppose $A$ is an AF algebra and $H$ is a separable $C^*$\=/correspondence over $A$.  Then the following are equivalent:
\begin{enumerate}
  \item $\mathcal{O}_A(H)$ is AF--embeddable;
  \item $\mathcal{O}_A(H)$ is quasidiagonal;
  \item $\mathcal{O}_A(H)$ is stably finite.
\end{enumerate}
Moreover, if $H$ is row-finite and faithful, the above conditions are equivalent to
\begin{enumerate}\setcounter{enumi}{3}
  \item if $x \in K_0(A)$ and $[H](x) \leq x$, then $[H](x) = x$.
\end{enumerate}
\end{theoremC}

We will prove the theorem by using Theorems A and B to write $\mathcal{O}_A(H)$ as a full corner of a crossed product $B \rtimes \mathbb{Z}$ where $B$ is AF.  Theorem C then follows by applying Brown's Embedding Theorem to $B \rtimes \mathbb{Z}$.  In particular, Theorem B implies $B := \mathcal{O}_A(H) \rtimes_\gamma \mathbb{T}$ is AF.  By Takai duality, it follows that $\mathcal{O}_A(H)$ is a full corner of $B \rtimes_{\hat{\gamma}} \mathbb{Z}$.  Brown's Embedding Theorem also characterizes the AF\=/emeddability of crossed products in terms of $K$\=/theory.  If $H$ is row--finite and faithful, then our $K$\=/theoretic calculation in Theorem A allows to rephrase condition (4) in terms of the $K$\=/theory of the dynamical system $(B, \hat{\gamma})$.

The paper is organized as follows.  In Section \ref{sec:CuntzPimsnerAlgebras} we recall some of the basic material on Cuntz--Pimsner algebras that will be needed in this paper and set up some notation that will be used throughout.  Section \ref{sec:CrossedProductCorrespondences} outlines the work of Hao and Ng in \cite{HaoNg} on crossed products $C^*$\=/correspondences by locally compact groups.  Section \ref{sec:KTheory} outlines Exel's picture of $K$\=/theory in terms of Fredholm operators (see \cite{ExelFredholmOperators}) and using Exel's techniques, we show that the order structure on $K_0$ can also be defined in terms of Fredholm operators and every row\=/finite $C^*$\=/correspondence induces a homomorphism on $K$\=/theory that preserves the order structure on $K_0$.

The main results start in Section \ref{sec:SkewProduct} where we use the Hao--Ng isomorphism to give an inductive limit decomposition of the crossed product $\mathcal{O}_A(H) \rtimes_\gamma \mathbb{T}$ and use this to prove Theorem A.  Section \ref{sec:AFAlgebras} is devoted to proving Theorem B.  Finally Section \ref{sec:BrownsTheorem} contains some applications to AF-Embeddability and in particular, we prove Theorem C in this section.

\begin{notation}
AF\=/algebras are always assumed to be separable.  All other $C^*$-algebras are not assumed to be separable.  Throughout, we use the notation $AB$ to denote the closed span of set $\{ a b : a \in A, b \in B \}$.  Notation such as $\langle H, H \rangle$, $A p A$, etc.\ is used similarly.
\end{notation}

\section{Cuntz-Pimsner Algebras}\label{sec:CuntzPimsnerAlgebras}

In this section we recall some preliminary results and set up some notation.  Most of the results in this section are taken from \cite{KatsuraCorr1} and \cite{KatsuraCorr2}.

Suppose $A$ is a $C^*$\=/algebra.  A \emph{(right) pre--Hilbert $A$\=/module} is a right $A$\=/module $H$ together with an \emph{$A$\=/valued inner product} $\langle \cdot, \cdot \rangle: H \times H \rightarrow A$ such that
\begin{enumerate}
  \item $\langle \xi, \xi \rangle \geq 0$ for every non--zero $\xi \in H$ with equality if and only if $\xi = 0$,
  \item $\langle \xi, \cdot \rangle : H \rightarrow A$ is $A$\=/linear for every $\xi \in H$
  \item $\langle \xi, \eta \rangle^* = \langle \eta, \xi \rangle$ for every $\xi, \eta \in H$.
\end{enumerate}
Define a norm on $H$ by $\|\xi\| = \|\langle \xi, \xi \rangle\|^{1/2}$ for every $\xi \in H$.  We say $H$ is a \emph{Hilbert $A$\=/module} if $H$ is complete with respect to this norm.

Unlike with Hilbert spaces, a bounded $A$\=/linear operator between Hilbert modules, need not have an adjoint.  Thus we will have to assume the existence of adjoints.  If $H$ and $K$ are Hilbert $A$\=/modules, an operator $T: H \rightarrow K$ is called \emph{adjointable} if there is an operator $T^*: K \rightarrow H$ such that
\[ \langle T\xi, \eta \rangle = \langle \xi, T^* \eta \rangle \qquad \text{ for every $\xi, \eta \in H$.} \]
In this case $T^*$ is unique, $T^*$ is adjointable with $T^{**} = T$, and $T$ and $T^*$ are both bounded and $A$\=/linear.  Moreover, the collection $\mathbb{B}(H, K)$ of all adjointable operators form $H$ to $K$ is a Banach space and $\mathbb{B}(H) := \mathbb{B}(H, H)$ is a $C^*$\=/algebra with the obvious operations.

Given $\xi, \eta \in H$, we define $\theta_{\xi, \eta}: H \rightarrow H$ by $\theta_{\xi, \eta}(\zeta) = \xi \langle \eta, \zeta \rangle$ for every $\zeta \in H$.  Then $\theta_{\xi, \eta} \in \mathbb{B}(H)$ and the following hold:
\begin{enumerate}
  \item $\theta_{\xi, \eta}^* = \theta_{\eta, \xi}$,
  \item If $T \in \mathbb{B}(H)$, then $T\theta_{\xi, \eta} = \theta_{T\xi, \eta}$ and $\theta_{\xi, \eta} T = \theta_{\xi, T^* \eta}$.
  \item $\theta_{\xi + \xi', \eta} = \theta_{\xi, \eta} + \theta_{\xi', \eta}$, and
  \item $\theta_{\xi, \eta + \eta'} = \theta_{\xi, \eta} + \theta_{\xi, \eta'}$.
\end{enumerate}
By the above calculations,
\[ \mathbb{K}(H) = \overline{\operatorname{span}} \, \{ \theta_{\xi, \eta} : \xi, \eta \in H \} \subseteq \mathbb{B}(H). \]
is an ideal of $\mathbb{B}(H)$.  The elements of $\mathbb{K}(H)$ are called the \emph{compact operators} on $H$.

Suppose $A$ and $B$ are $C^*$\=/algebras.  An \emph{$A$--$B$ $C^*$\=/correspondence} is a Hilbert $B$\=/module $H$ together with a *\=/homomorphism $\lambda: A \rightarrow \mathbb{B}(H)$.  Then $H$ is an \mbox{$A$--$B$ bimodule} with $a \xi := \lambda(a)(\xi)$ for every $a \in A$ and $\xi \in H$.  The case when $A = B$ is of special interest.  In this case, we say $H$ is a \emph{$C^*$\=/correspondence over $A$}.  We will often suppress the map $\lambda$ and write $\lambda(a)(\xi) = a \xi$ for $a \in A$ and $\xi \in H$.

An $A$--$B$ $C^*$\=/correspondence $H$ is called \emph{faithful} if $H$ is faithful as a left module; that is, the left action $\lambda: A \rightarrow \mathbb{B}(H)$ is injective.  We say $H$ is \emph{row--finite} if $\lambda(A) \subseteq \mathbb{K}(H)$.  The conditions row--finite and faithful $C^*$\=/correspondence analogues of graphs that are row--finite with no sources.

If $H$ and $K$ are $A$--$B$ $C^*$\=/correspondences, then $H \oplus K$ is an $A$--$B$ $C^*$\=/correspondence in an obvious way.  If $H_\alpha$ is a collection of $A$--$B$ $C^*$\=/correspondences, we define
\[ \bigoplus_\alpha H_\alpha = \{ (\xi_\alpha)_\alpha : \xi_\alpha \in H_\alpha, \, \sum_\alpha \langle \xi_\alpha, \xi_\alpha \rangle \text{ converges in norm in $A$ } \}. \]
Then $\bigoplus_\alpha H_\alpha$ is an $A$--$B$ $C^*$\=/correspondence with the obvious bimodule structure and the inner product given by
\[ \langle (\xi_\alpha)_\alpha, (\eta_\alpha)_\alpha \rangle = \sum_\alpha \langle \xi_\alpha, \eta_\alpha \rangle, \]
where the sum converges in norm.

There is also a tensor product of $C^*$\=/correspondences which we now describe.  Suppose $H$ is an $A$--$B$ $C^*$\=/correspondence and $K$ is a $B$--$C$ $C^*$\=/correspondence.  Let $H \odot_B K$ denote the algebraic tensor product of $H$ and $K$ over $B$.  Define a $C$\=/valued inner product on $H \odot_B K$ by
\[ \langle \xi \otimes \eta, \xi' \otimes \eta' \rangle = \langle \eta, \langle \xi, \xi' \rangle \eta' \rangle \]
for $\xi, \xi' \in H$ and $\eta, \eta' \in K$, and let $H \otimes_B K$ denote the completion of $H \odot_B K$.  Then $H \otimes_B K$ is an $A$--$C$ $C^*$\=/correspondence.  If $H$ is a $C^*$\=/correspondence over $A$, define $H^{\otimes{0}} = A$ and $H^{\otimes (k + 1)} = H^{\otimes k} \otimes_A H$ for every $k \geq 0$.  Then define the \emph{Fock space} of $H$ to be the $C^*$\=/correspondence
\[ \mathcal{F}_A(H) = A \oplus H \oplus H^{\otimes 2} \oplus \cdots \]
over $A$.

\begin{example}\label{eg:GeneralizedMorphism}
Suppose $A$ and $B$ are $C^*$\=/algebras and $\alpha: A \rightarrow M(B)$ is a *\=/homomorphism.  Let $H_\alpha = B$ and note that $\mathbb{B}(H_\alpha) = M(B)$.  Hence we may view $H$ is an $A$--$B$ $C^*$\=/correspondence.  In this way, we may think of $C^*$\=/correspondences as being a generalizations of *\=/homomorphisms between $C^*$\=/algebras.  The tensor product of $C^*$\=/correspondences defined above generalizes the composition of *\=/homomorphisms.
\end{example}

We will now associate a $C^*$\=/algebra to a $C^*$\=/correspondence.  For a $C^*$\=/correspondence $H$ over a $C^*$\=/algebra $A$, we will define an algebra $\mathcal{O}_A(H)$ which can be thought of as a crossed product by the ``generalized *\=/homomorphism'' $H$ as in example \ref{eg:GeneralizedMorphism}.

A \emph{Toeplitz representation} $(\pi, \tau): (A, H) \rightarrow B$ consists of a $C^*$\=/algebra $B$, a *\=/homomorphism $\pi: A \rightarrow B$, and a linear map $\tau: H \rightarrow B$ such that
\[ \pi(a) \tau(\xi) = \tau(a \xi) \qquad \text{ and } \qquad \tau(\xi)^* \tau(\eta) = \pi(\langle \xi, \eta \rangle) \]
for every $a \in A$ and $\xi \in H$.  We note if $(\pi, \tau)$ is a Toeplitz representation, then $\tau(\xi) \pi(a) = \tau(\xi a)$.  Moreover, $\tau$ is automatically contractive and $\tau$ is isometric whenever $\pi$ is injective.

There is always a Toeplitz representation $(\pi, \tau): (A, H) \rightarrow \mathcal{T}(H)$ which is universal in the sense that given any other Toeplitz representation $(\pi', \tau') : (A, H) \rightarrow B$, there is a unique *\=/homomorphism $\psi : \mathcal{T}(H) \rightarrow B$ such that $\psi \circ \pi = \pi'$ and $\psi \circ \tau = \tau'$.
\[ \begin{tikzcd} A \arrow{r}{\pi} \arrow[bend right]{ddr}[swap]{\pi'} & \mathcal{T}(H) \arrow[dotted]{dd}{\exists\hskip .5pt ! \,\psi} & H \arrow{l}[swap]{\tau} \arrow[bend left]{ddl}{\tau'} \\ & & \\ & B & \end{tikzcd} \]
The $C^*$\=/algebra $\mathcal{T}(H)$ is called the \emph{Toeplitz--Pimsner algebra}.  By the usual argument, $\mathcal{T}(H)$ is unique up to a canonical isomorphism.  Moreover, $\mathcal{T}(H)$ is generated as a $C^*$\=/algebra by $\pi(A)$ and $\tau(H)$.

To define $\mathcal{O}_A(H)$, we need to impose an extra condition on the Toeplitz representations $(\pi, \tau)$.  The following proposition is Lemma 2.2 in \cite{KPW}.

\begin{proposition}\label{prop:InducedMapOnCompacts}
Given a Toeplitz representation $(\pi, \tau) : (A, H) \rightarrow B$, there is a unique *\=/homomorphism $\varphi : \mathbb{K}(H) \rightarrow B$ such that $\varphi(\theta_{\xi, \eta}) = \tau(\xi) \tau(\eta)^*$ for every $\xi, \eta \in H$.  If $\pi$ is injective, then so is $\varphi$.  Moreover, for every $a \in A$, $\xi \in H$, and $k \in \mathbb{K}(H)$,
\[ \varphi(k)\tau(\xi) = \tau(k(\xi)), \quad \varphi(k) \pi(a) = \varphi(k \lambda(a)), \quad \pi(a) \varphi(k) = \varphi(\lambda(a) k). \]
\end{proposition}

Define an ideal $J_H \subseteq A$ by $J_H = \lambda^{-1}(\mathbb{K}(H)) \cap (\ker \lambda)^\perp$, where
\[(\ker \lambda)^\perp = \{ a \in A : ab = 0 \text{ for all } b \in \ker \lambda \}. \]
We say a Toeplitz representation $(\pi, \tau) : (A, H) \rightarrow B$ is \emph{covariant} if $\pi(a) = \varphi(\lambda(a))$ for every $a \in J_H$.  As with the Toeplitz\=/Pimsner algebra, there is always a universal covariant Toeplitz representation $(\pi, \tau) : (A, H) \rightarrow \mathcal{O}_A(H)$.  The algebra $\mathcal{O}_A(H)$ is called the \emph{Cuntz--Pimsner} algebra.  As before, $\mathcal{O}_A(H)$ is unique up to a canonical isomorphism and is generated as a $C^*$\=/algebra by $\pi(A)$ and $\tau(H)$.  Moreover,
\[ \mathcal{O}_A(H) = \overline{\operatorname{span}} \, \{\tau(\xi_1) \cdots \tau(\xi_n) \pi(a) \tau(\eta_m)^* \cdots \tau(\eta_1)^* : \xi_i, \eta_i \in H, a \in A\}. \]

\begin{remark}
A more explicit description of $\mathcal{T}_A(H)$ and $\mathcal{O}_A(H)$ can be given as follows.  Given $a \in A$ and $\xi \in H$, define $\pi(a), \tau(a) \in \mathbb{B}(\mathcal{F}_A(H))$ by
\begin{align*}
   \pi(a)(\eta_1 \otimes \cdots \otimes \eta_n) &= a \eta_1 \otimes \cdots \otimes \eta_n \\
\intertext{and}
   \tau(\xi)(\eta_1 \otimes \cdots \otimes \eta_n) &= \xi \otimes \eta_1 \otimes \cdots \otimes \eta_n.
\end{align*}
The $\mathcal{T}_A(H)$ is isomorphic to the $C^*$\=/subalgebra of $\mathbb{B}(\mathcal{F}_A(H))$ generated by the $\pi(a)$ and $\tau(\xi)$ for $a \in A$ and $\xi \in H$.  The algebra $\mathcal{O}_A(H)$ is roughly the quotient of $\mathcal{T}_A(H)$ by the compact operators on $\mathcal{F}_A(H)$ (see Section 4 of \cite{KatsuraCorr2}).
\end{remark}

\begin{example}\label{eg:GeneralizedCrossedProduct}
Suppose $A$ is a $C^*$\=/algebra and $\alpha \in \operatorname{Aut}(A)$.  Define a $C^*$\=/correspondence $H$ over $A$ as in Example \ref{eg:GeneralizedMorphism}.  Then $\mathcal{O}_A(H) \cong A \rtimes_\alpha \mathbb{Z}$.  In general, we may think of $\mathcal{O}_A(H)$ as a crossed product of $A$ by the generalized morphism $H$.
\end{example}

Suppose $(\pi, \tau): (A, H) \rightarrow B$ is a Toeplitz representation on a $C^*$\=/algebra $B$.  We say $(\pi, \tau)$ \emph{admits of gauge action} if there is a point\=/norm continuous group homomorphism $\gamma: \mathbb{T} \rightarrow \operatorname{Aut}(B)$ such that $\gamma_z(\pi(a)) = \pi(a)$ and $\gamma_z(\tau(\xi)) = z \tau(\xi)$ for every $a \in A$, $\xi \in H$, and $z \in \mathbb{T}$.  It is an easy consequence of the universal property (definition) of $\mathcal{O}_A(H)$ that the universal covariant representation admits a gauge action.  In some sense, the universal covariant representation is the only injective covariant representation which admits a gauge action.

\begin{theorem}[Gauge Invariant Uniqueness Theorem]\label{thm:GIUT}
Let $(\pi, \tau): (A, H) \rightarrow B$ be a covariant Toeplitz representation and let $\psi: \mathcal{O}_A(H) \rightarrow B$ denote the induced map.  If $(\pi, \tau)$ admits a gauge action and $\pi$ is injective, then $\psi$ is injective.
\end{theorem}

We end this section with a technical proposition about non\=/degenerate correspondences.  A $C^*$\=/correspondence $(A, H)$ is said to be \emph{non\=/degenerate} if $AH = H$. Occasionally, some of our arguments will require non\=/degeneracy, but fortunately, up to Morita equivalence, we can always reduce to this case.

\begin{proposition}\label{prop:NonDegenerate}
Suppose $A$ is a $C^*$\=/algebra and $H$ is a $C^*$\=/correspondence over $A$.  Then there is a non\=/degenerate $C^*$\=/correspondence $H'$ over $A$ such that $\mathcal{O}_A(H)$ is Morita equivalent to $\mathcal{O}_A(H')$.
\end{proposition}

\begin{proof}
Note that $H' := AH$ is a non\=/degenerate $C^*$\=/correspondence over $A$.  Let $\lambda': A \rightarrow \mathbb{B}(H')$ denote the left multiplication map.  Let $(\pi, \tau) : (A, H) \rightarrow \mathcal{O}_A(H)$ be the universal covariant representation.  Let $\pi' = \pi$ and $\tau' = \tau|_{H'}$.  It is easy to verify $(\pi', \tau') : (A', H') \rightarrow \mathcal{O}_A(H')$ is a Toeplitz representation.

We claim $(\pi', \tau')$ is covariant.  Clearly $\ker \lambda \subseteq \ker \lambda'$ and hence $(\ker \lambda')^\perp \subseteq (\ker \lambda)^\perp$.  Suppose $a \in A$ is such that $\lambda'(a) \in \mathbb{K}(H')$ and suppose $\varepsilon > 0$.  Then there are vectors $\xi_1, \ldots, \xi_n, \eta_1, \ldots, \eta_n \in H'$ such that
\[ \left\| \lambda'(a) - \sum \theta_{\xi_i, \eta_i} \right\|_{\mathbb{B}(H')} < \varepsilon. \]
Replacing $\xi_i$ with $\xi_i / \|\xi_i\|$ and replacing $\eta_i$ with $\|\xi_i\| \eta_i$, we may assume $\| \xi_i \| = 1$ for every $i$.  Choose a self\=/adjoint contraction $e \in A$ such that $\| a e - a \| < \varepsilon$ and $\| e \eta_i - \eta_i \| < \varepsilon/n$.  Then in the operator norm on $\mathbb{B}(H)$, we have
\begin{align*}
 & \hphantom{=} \left\| \lambda(a) - \sum \theta_{\xi_i, \eta_i} \right\| \\
 &\leq \|\lambda(a) - \lambda(ae) \| + \left\| \lambda(ae) - \sum \theta_{\xi_i, e \eta_i} \right\| + \sum \| \theta_{\xi_i, e \eta_i}-\theta_{\xi_i, \eta_i}\|.
\end{align*}
The first and third term in the above sum are bounded by $\varepsilon$.  This also holds for the middle term since for every $\zeta \in H$ with $\| \zeta \| \leq 1$, $e \zeta \in H'$ and hence
\[ \left\| (\lambda(ae) - \sum \theta_{\xi_i, e \eta_i})(\zeta) \right\| = \left\| (\lambda'(a) - \sum \theta_{\xi_i, \eta_i})(e \zeta) \right\|<\varepsilon. \]
Therefore, $\lambda(a) \in \mathbb{K}(H)$ and hence $J_{H'} \subseteq J_H$.  The calculation above also shows $\varphi'(\lambda'(a)) = \varphi(\lambda(a))$ for every $a \in J_{H'}$.  Hence the covariance of $(\pi', \tau')$ follows from the covariance of $(\pi, \tau)$.

Let $\psi: \mathcal{O}_A(H') \rightarrow \mathcal{O}_A(H)$ be the *\=/homomorphism induced by $(\pi', \tau')$.  By the Gauge Invariance Uniqueness Theorem, $\psi$ is injective.  Moreover,
\[ \psi(\mathcal{O}_A(H')) = \pi(A) \mathcal{O}_A(H) \pi(A). \]
Since $\mathcal{O}_A(H)$ is generated as an ideal by $\pi(A)$, it follows that $\mathcal{O}_A(H')$ is a full, hereditary subalgebra of $\mathcal{O}_A(H)$.  In particular, $\mathcal{O}_A(H')$ and $\mathcal{O}_A(H)$ are Morita equivalent.
\end{proof}

\begin{remark}
There is also a version of Proposition \ref{prop:NonDegenerate} for the Toeplitz\=/Pimsner algebra $\mathcal{T}_A(H)$ (see the paragraph preceding Lemma C.17 in \cite{KatsuraCorr2}).
\end{remark}

\section{Crossed Product Correspondences}\label{sec:CrossedProductCorrespondences}

In this section, we recall some of the results from \cite{HaoNg}.  In particular, suppose $H$ is a $C^*$\=/correspondence over $A$ and $G$ is a locally compact group acting continuously on the correspondence $(A, H)$ (see below), there is a crossed product correspondence $H \rtimes G$ over the full crossed product $A \rtimes G$.  Moreover, the actions induces a continuous action of $G$ on $\mathcal{O}_A(H)$.  If $G$ is amenable, then Hao and Ng show there is an isomorphism
\[ \mathcal{O}_{A \rtimes G}(H \rtimes G) \cong \mathcal{O}_A(H) \rtimes G. \]

In this section, we will outline the construction of this isomorphism.  In the case where $G$ is abelian, we will show there is a canonical continuous action of $\hat{G}$ on both algebras and the Hao--Ng isomorphism is $\hat{G}$\=/equivariant.  In Section \ref{sec:SkewProduct}, these results will be applied to the gauge action $\gamma$ on a Cuntz--Pimsner algebra $\mathcal{O}_A(H)$ to give an alternate description of the crossed product $\mathcal{O}_A(H) \rtimes_\gamma \mathbb{T}$.

Suppose $A$ is a $C^*$\=/algebra, $H$ is a $C^*$\=/correspondence over $A$, and $G$ is a locally compact group.  A \emph{correspondence action} is a pair
\[ (\alpha, \beta) : G \rightarrow \operatorname{Aut}(A, H) \]
where $\alpha$ is an action of $G$ on $A$ as *\=/homomorphisms and $\beta$ is an action of $G$ on $H$ as isometric $\mathbb{C}$\=/linear isomorphisms such that for every $a \in A$ and $\xi \in H$, the maps $G \rightarrow A$ and $G \rightarrow H$ given by $s \mapsto \alpha_s(a)$ and $s \mapsto \beta_s(\xi)$ are norm continuous and for every $a \in A$, $\xi, \eta \in H$, and $s \in G$, the following hold:
\begin{enumerate}
  \item $\beta_s(\xi) \alpha_s(a) = \beta_s(\xi a)$,
  \item $\alpha_s(a) \beta_s(\xi) = \beta_s(a \xi)$, and
  \item $\alpha_s(\langle \xi, \eta \rangle) = \langle \beta_s(\xi), \beta_s(\eta) \rangle$.
\end{enumerate}

Given $a \in C_c(G, A)$, $\xi \in C_c(G, H)$, and $s \in G$, we define
\begin{align*}
(a \xi)(s) &= \int_G a(t) \beta_t( \xi(t^{-1} s)) \, dt, \\
(\xi a)(s) &= \int_G \xi(t) \alpha_t( a(t^{-1} s)) \, dt, \quad \text{and} \\
\langle \xi, \eta \rangle(s) &= \int_G \alpha_t( \langle \xi(t), \eta(ts) \rangle) \, dt,
\end{align*}
where the integrals are taken with respect to the Haar measure on $G$.  Taking completions, we obtain a $C^*$\=/correspondence $H \rtimes_\beta G$ over the full crossed product $A \rtimes_\alpha G$.

Let $H$ be a $C^*$\=/correspondence over $A$, let $(\pi, \tau) : (A, H) \rightarrow \mathcal{O}_A(H)$ be the universal covariant representation and suppose  $(\alpha, \beta)$ is a correspondence action $G$ of $(A, H)$.  For $s \in G$, define $\pi_s : A \rightarrow \mathcal{O}_A(H)$ and $\tau_s : H \rightarrow \mathcal{O}_A(H)$ by $\pi_s(a) = \pi(\alpha_s(a))$ and $\tau_s(\xi) = \tau(\beta_s(\xi))$.  Then $(\pi_s, \tau_s)$ is a covariant Toeplitz representation and hence induces a *\=/homomorphism $\gamma_s : \mathcal{O}_A(H) \rightarrow \mathcal{O}_A(H)$.  Then the map $\gamma : G \rightarrow \operatorname{Aut}(\mathcal{O}_A(H))$, $s \mapsto \gamma_s$, is a group homomorphism.  Moreover, $\gamma$ is strongly continuous in the sense that the map $G \rightarrow \mathcal{O}_A(H)$ given by $s \mapsto \gamma_s(x)$ is norm continuous for every $x \in \mathcal{O}_A(H)$.

The following theorem is the main result of \cite{HaoNg}.

\begin{theorem}\label{thm:HaoNg}
Suppose $(\alpha, \beta)$ is a correspondence action of a locally compact group $G$ on a $C^*$\=/correspondence $(A, H)$.  Let $(\pi, \tau) : (A, H) \rightarrow \mathcal{O}_A(H)$ be the universal covariant representation.  If $G$ is amenable, then the maps
\[ \pi' : C_c(G, A) \rightarrow C_c(G, \mathcal{O}_A(H)) \quad \text{and} \quad \tau' : C_c(G, H) \rightarrow C_c(G, \mathcal{O}_A(H)) \]
given by $\pi'(a)(s) = \pi(a(s))$ and $\tau'(\xi)(s) = \tau(\xi(s))$ induce an isomorphism
\[ \mathcal{O}_{A \rtimes_\alpha G}(H \rtimes_\beta G) \rightarrow \mathcal{O}_A(H) \rtimes_\gamma G. \]
\end{theorem}

We now specialize to the case where $G$ is abelian.  Let $H$ be a $C^*$\=/correspondence over a $C^*$\=/algebra $A$ and let $(\alpha, \beta)$ be a correspondence action of $G$ on $(A, H)$.  Given $\chi \in \hat{G}$, define $\hat{\alpha}_\chi : C_c(G, A) \rightarrow C_c(G, A)$ by $\hat{\alpha}_\chi(a)(s) = \chi(s)^{-1} a(s)$ and $\hat{\tau}_\chi : C_c(G, H) \rightarrow C_c(G, H)$ by $\hat{\tau}_\chi(\xi)(s) = \chi(s)^{-1} \xi(s)$.  Then $(\hat{\alpha}, \hat{\tau})$ defines a correspondence action on $(A \rtimes_\alpha G, H \rtimes_\beta G)$.  Now $(\hat{\alpha}, \hat{\beta})$ defines a strongly continuous actions on $\delta$ on $\mathcal{O}_{A \rtimes_\alpha G}(H \rtimes_\beta G)$.  Let $\hat{\gamma}$ denote the action on $\mathcal{O}_A(H) \rtimes_\gamma G$ dual to the action of $\gamma$; that is, $\hat{\gamma}$ is given by $\hat{\gamma}_\chi(x)(s) = \chi(s)^{-1} x(s)$ for every $x \in C_c(G, \mathcal{O}_A(H))$, $\chi \in \hat{G}$, and $s \in G$.

\begin{proposition}\label{prop:equivariant}
With the notation above, the isomorphism
\[ \mathcal{O}_{A \rtimes_\alpha G}(H \rtimes_\beta G) \rightarrow \mathcal{O}_A(H) \rtimes_\gamma G \]
is $\hat{G}$\=/equivariant.
\end{proposition}

\begin{proof}
We retain the notation from Theorem \ref{thm:HaoNg}.  Given $a \in C_c(G, A)$, $s \in G$, and $\chi \in \hat{G}$, we have
\[ \pi'(\hat{\alpha}_\chi(a))(s) = \pi(\hat{\alpha}_\chi(a)(s)) = \pi(\chi(s)^{-1} a(s)) = \chi(s)^{-1} \pi'(a)(s) = \hat{\gamma}_\chi(\pi'(a))(s). \]
Hence $\pi' \circ \hat{\alpha}_\chi = \hat{\gamma}_\chi \circ \pi'$.  Similarly, $\tau' \circ \hat{\alpha}_\chi = \hat{\gamma}_\chi \circ \tau'$ and the result follows.
\end{proof}

\section{Fredholm Operators and $K$\=/Theory}\label{sec:KTheory}

In this section we outline the main results from \cite{ExelFredholmOperators} on the relationship between $K$\=/theory and Fredholm operators on Hilbert modules.  The utility of defining $K$\=/theory in terms of Fredholm operators is that $C^*$\=/correspondences will induce maps on $K$\=/theory by using a tensor product construction.  If fact Exel's motivation for this approach was to give a constructive proof that Morita equivalent $C^*$\=/algebras have the same $K$\=/theory.  In fact, if $H$ is an $A$--$B$ imprimitivity bimodule, then $H$ induces an isomorphism $[H] : K_*(A) \rightarrow K_*(B)$.  We will use Exel's techniques to show any row--finite $A$--$B$ $C^*$\=/correspondence $H$ induces a homomorphism $[H] : K_*(A) \rightarrow K_*(B)$ that preserves the order structure on $K_0$.

\begin{definition}
Suppose $A$ is a $C^*$\=/algebra and $M$ and $N$ are Hilbert $A$\=/modules.  An adjointable operator $T : M \rightarrow N$ is called \emph{$A$\=/Fredholm}, or just \emph{Fredholm}, if there is an adjointable operator $S : N \rightarrow M$ such that $1 - TS$ and $1 - ST$ are compact.  We say $T$ is \emph{regular} if in addition we can choose $S$ with $STS = S$ and $TST = T$.
\end{definition}

\begin{definition}
We say a Hilbert $A$\=/module $M$ has \emph{finite rank} if the identity operator $1_M$ is a compact operator on $M$.
\end{definition}

It is a well known fact that for a unital $C^*$\=/algebra $A$, the group $K_0(A)$ can be viewed as the Grothendieck group of the isomorphism classes of finitely generated Hilbert $A$\=/modules. See \cite[Exercise 15.K]{Wegge-Olsen} for an outline of the proof.  If $A$ is a (not necessarily unital) $C^*$\=/algebra and $M$ is a finitely generated Hilbert $A$\=/module, then $M$ is also a finitely generated Hilbert $\tilde{A}$--module, where $\tilde{A}$ is the unitization of $A$.

\begin{definition}\label{defn:RankOfModule}
If $A$ is a $C^*$\=/algebra and $M$ is a finite rank Hilbert $A$\=/module, the class of $M$ in $K_0(A)$ will be denoted $\operatorname{rank}(M)$.
\end{definition}

\begin{proposition}[Proposition 3.3 in \cite{ExelFredholmOperators}]
Suppose $T : M \rightarrow N$ is a regular Fredholm operator between Hilbert $A$\=/modules.  Then $\ker(T)$ and $\ker(T^*)$ are both finite rank Hilbert $A$\=/modules.
\end{proposition}

The following definition is a Hilbert module version of the classical Fredholm index.

\begin{definition}
If $T$ is a regular Fredholm operator between Hilbert $A$\=/modules, define $\operatorname{ind}(T)$ in $K_0(A)$ by
\[ \operatorname{ind}(T) = \operatorname{rank}(\ker{T}) - \operatorname{rank}(\ker(T^*)). \]
\end{definition}

\begin{remark}
The Fredholm index can be defined for all Fredholm operators.  The idea is that every Fredholm operator is ``stably regular.''  That is, if $T : M \rightarrow N$ is Fredholm, then there is a regular Fredholm operator $\tilde{T} : M \oplus A^n \rightarrow N \oplus A^n$ extending $T$.  Then $\operatorname{ind}(T)$ is defined to be $\operatorname{ind}(\tilde{T})$.  See Proposition 3.8 and Lemma 3.9 in \cite{ExelFredholmOperators} for the details.
\end{remark}

Let $F_0(A)$ denote the set of all Fredholm operators on Hilbert $A$\=/modules.  (To avoid set\=/theoretic issues, we should fix some sufficiently large Hilbert $A$\=/module $H$ and only consider Fredholm operators whose domain and codomain are submodules of $H$.  However, we will ignore these issues in what follows.)  $F_0(A)$ forms a unital, abelian semigroup with addition given by direct sum.  We define an equivalence relation on $F_0(A)$ by $S \sim T$ if there is a positive integer $n$ such that $S \oplus T^* \oplus 1_{A^n}$ is a compact perturbation of an invertible operator.  Let $F(A)$ denote the quotient $F_0(A) / \sim$ and given $S \in F_0(A)$, let $[S]$ denote that class of $S$ in $F(A)$.

\begin{theorem}[Corollary 3.17 in \cite{ExelFredholmOperators}]
The set $F(A)$ is an abelian group with $[S] + [T] = [S \oplus T]$ and $-[S] = [S^*]$.  Moreover, the Fredholm index induces a group isomorphism
\[ \operatorname{ind} : F(A) \rightarrow K_0(A). \]
\end{theorem}

For our purposes, it will also be useful to describe the positive cone $K_0(A)^+$ in terms of Hilbert $A$\=/modules.

\begin{theorem}\label{thm:PositiveCone}
Suppose $x \in K_0(A)$.  Then $x \geq 0$ if and only if there is a surjective, regular, Fredholm operator $T \in F_0(A)$ such that $\operatorname{ind}(T) = x$.
\end{theorem}

\begin{proof}
If $x \geq 0$, there is a finitely generated Hilbert $A$\=/module such that $\operatorname{rank}{M} = x$.  The operator $T : M \rightarrow 0$ satisfies the conditions.  Conversely, suppose $T : M \rightarrow N$ is given as above.  Choose $S : N \rightarrow M$ such that $1 - ST$ and $1 - TS$ are compact, $TST = T$ and $STS = S$.  Since $TST = T$ and $T$ is surjective, $TS = 1$.  Thus $S$ is injective.  By Proposition 3.5(ii) in \cite{ExelFredholmOperators},
\[ \operatorname{rank}(\ker(T^*)) = \operatorname{rank}(\ker(S)) = 0, \]
and hence $x = \operatorname{ind}(T) = \operatorname{rank}(\ker(T)) \geq 0$.
\end{proof}

\begin{proposition}\label{prop:CorrKTheory}
Suppose $A$ and $B$ are $C^*$\=/algebras and $H$ is an $A$--$B$ $C^*$\=/correspondence such that $\lambda(a)$ is compact for every $a \in A$.  If $T \in F_0(A)$, then $T \otimes 1_H \in F_0(B)$ and the map $T \mapsto T \otimes 1_H$ induces a positive group homomorphism $[H] : K_0(A) \rightarrow K_0(B)$.
\end{proposition}

\begin{proof}
Suppose $T : M \rightarrow N$ is a Fredholm operator and choose $S : N \rightarrow M$ with $1 - ST$ and $1 - TS$ compact.  Since $\lambda(A) \subseteq \mathbb{K}_B(H)$, $(1 - ST) \otimes 1_H$ and $(1 - TS) \otimes 1_H$ are compact.  It follows that $T \otimes 1_H : M \otimes_A H \rightarrow N \otimes_A H$ is Fredholm.

Now suppose $S, T \in F_0(A)$ are such that $S \tilde T$.  Then there is a positive integer $n$ and a compact operator $R$ such that
\[ S \oplus T^* \oplus 1_{A^n} + R \]
is invertible.  Therefore,
\[ (S \otimes 1_H) \oplus (T^* \otimes 1_H) \oplus (1_{A^n} \otimes 1_H) + R \otimes 1_H = (S \oplus T^* \oplus 1_{A^n} + R) \otimes 1_H \]
is invertible.  Since $R \otimes 1_H$ is compact and $1_{A^n} \otimes 1_H$ has index 0, we have
\[ \operatorname{ind}(S \otimes 1_H) = \operatorname{ind}(T \otimes 1_H). \]
This proves $T \mapsto T \otimes 1_H$ induces a group homomorphism $[H] : K_0(A) \rightarrow K_0(B)$.

It remains to show $[H]$ is positive.  However, this follows from Theorem \ref{thm:PositiveCone} since if $T$ is a surjective, regular, Fredholm operator, then so is $T \otimes 1_H$.
\end{proof}

\begin{remark}
The argument above is same construction and proof that Exel uses in \cite{ExelFredholmOperators} to build a map $[H] : K_0(A) \rightarrow K_0(B)$.  Since Exel was concerned mostly with understanding imprimitivity bimodules, $H$ was assumed to be a left--full Hilbert bimodule.  However, this assumption was only used to show if $T : M \rightarrow N$ is a compact operator between Hilbert $A$\=/modules, then $T \otimes_A H$ is compact.  Since this also holds when $H$ is row--finite, the proof also holds in our setting.
\end{remark}

Any row--finite $C^*$\=/correspondence also induces a group homomorphism on $K_1$.  To see this, recall $K_1(A) = K_0(SA)$, where $SA := C_0(\mathbb{R}) \otimes A$ is the suspension of $A$.  If $H$ is a row--finite $A$--$B$ $C^*$\=/correspondence, then $SH := C_0(\mathbb{R}) \otimes H$ is a row--finite $SA$--$SB$ $C^*$\=/correspondence in the obvious way.  Hence $SH$ induces a homomorphism on $K_0$ that can be viewed a homomorphism $[H] : K_1(A) \rightarrow K_1(B)$.

\begin{proposition}\label{prop:CorrKTheoryProperties}
Let $A$, $B$, and $C$ be $C^*$\=/algebras.
\begin{enumerate}
  \item If $H$ and $K$ are isomorphic row--finite $A$--$B$ $C^*$\=/correspondences, then
  \[ [H] = [K] : K_*(A) \rightarrow K_*(B). \]
  \item If $H$ is a row--finite $A$--$B$ $C^*$\=/correspondence and $K$ is a row--finite $B$--$C$ $C^*$\=/correspondence, then
  \[ [H \otimes_B K] = [K] \circ [H] : K_*(A) \rightarrow K_*(C). \]
  \item If $\varphi : A \rightarrow B$ is a morphism, then we may view $B$ as a row--finite $A$--$B$ $C^*$\=/correspondence as in Example \ref{eg:GeneralizedMorphism}.  Then $[B] = [\varphi] : K_*(A) \rightarrow K_*(B)$.
\end{enumerate}
\end{proposition}

\begin{proof}
The proofs are easy computations.  For example, we prove (i) in the case of $K_0$.  Suppose $H \cong K$ as correspondences.  Then there is a unitary $u \in \mathbb{B}(H, K)$ such that $\lambda(a) u = u \lambda(a)$ for every $a \in A$.  If $T : M \rightarrow N$ is a Fredholm operators between Hilbert $A$--modules, then the diagram
\begin{center}
\begin{tikzcd}[column sep = large]
M \otimes_A H \arrow{r}{T \otimes 1_H} \arrow{dd}{1_M \otimes u} & N \otimes_A H \arrow{dd}{1_N \otimes u} \\ \\ M \otimes_A K \arrow{r}{T \otimes 1_K} & N \otimes_A K
\end{tikzcd}
\end{center}
commutes.  Since $1_M \otimes u$ and $1_N \otimes u$ are invertible,
\[ \operatorname{ind}(T \otimes 1_H) = \operatorname{ind}((1_N \otimes u)(T \otimes 1_H)) = \operatorname{ind}((T \otimes 1_K)(1_M \otimes u)) = \operatorname{ind}(T \otimes 1_K) \]
by \cite[Proposition 3.11(ii)]{ExelFredholmOperators}.  Thus $[H](T) = [K](T)$.
\end{proof}

The following will be useful when working with $C^*$\=/correspondences that are not assumed to be non\=/degenerate.

\begin{proposition}\label{prop:NonDegKTheory}
If $H$ is a row--finite $A$--$B$ $C^*$\=/correspondence and $AH \subseteq K \subseteq H$ are isometric inclusions of $A$--$B$ correspondences, then $[K] = [H] : K_*(A) \rightarrow K_*(B)$.
\end{proposition}

\begin{proof}
If $M$ is a Hilbert $A$\=/module, there is an isometric embedding $M \otimes_A K \hookrightarrow M \otimes_A H$ given by $\xi \otimes \eta \mapsto \xi \otimes \eta$.  Suppose $\xi \in M$ and $\eta \in H$.  By Cohen factorization, we may write $\xi = \xi' a$ for some $\xi' \in M$ and $a \in H$.  Then
\[ \xi \otimes \eta = \xi' a \otimes \eta = \xi' \otimes a \eta \in M \otimes K \]
since $AH \subseteq K$.  As $M \otimes H$ is the closed span the $\xi \otimes \eta$ with $\xi \in M$ and $\eta \in H$, it follows that $M \otimes_A K = M \otimes_A H$.

Now given a Fredholm operator $T : M \rightarrow N$ between Hilbert $A$\=/modules, we have $T \otimes 1_K = T \otimes 1_H$ and it is clear that $[H](T) = [K](T)$.
\end{proof}

\begin{proposition}\label{prop:FaithfulKTheory}
Suppose $A$ and $B$ are unital $C^*$\=/algebras such that $B$ is stably finite and suppose $H$ is a row--finite $A$--$B$ $C^*$\=/correspondence.  If $H$ is faithful, then $[H] : K_0(A) \rightarrow K_0(B)$ is faithful.  That is, if $x \in K_0(A)$ is positive and $[H](x) = 0$, then $x = 0$.
\end{proposition}

\begin{proof}
Suppose $x \in K_0(A)$ is positive and $[H](x) = 0$.  Then there is a finitely generated Hilbert $A$\=/module $M$ such that $\operatorname{rank}(M) = x$.  Now, $M \otimes_A H$ is a finitely generated Hilbert $B$\=/module and $\operatorname{rank}(M \otimes_A H) = [H](x) = 0$.  Viewing $[H](x)$ as an element of $K_0(\tilde{B})$ where $\tilde{B}$ is the unitization of $B$, there is an integer $n$ such that $(M \otimes_A H) \oplus \tilde{B}^n \cong \tilde{B}^n$.  Since $\tilde{B}$ is stably finite, we have that $M \otimes_A H = 0$ (see Theorem 6.4 of \cite{Gipson}).

We claim $M = 0$.  Suppose to the contrary, there is a non--zero $\xi \in H$.  Then $\langle \xi, \xi \rangle \neq 0$ and since $H$ is faithful, there is an $\eta \in H$ such that $\langle \xi, \xi \rangle^{1/2} \eta \neq 0$.  But now,
\[ \langle \xi \otimes \eta, \xi \otimes \eta \rangle = \langle \eta, \langle \xi, \xi \rangle \eta \rangle = \langle \langle \xi, \xi \rangle^{1/2} \eta, \langle \xi, \xi \rangle^{1/2} \eta \rangle \neq 0. \]
Therefore, $\xi \otimes \eta$ is non--zero which contradicts the fact that $M \otimes_A H = 0$.
\end{proof}

\section{The Crossed Product by the Gauge Action}\label{sec:SkewProduct}

We apply the work of Hao and Ng described in Section \ref{sec:CrossedProductCorrespondences} to the gauge action on Cuntz--Pimsner algebras.  In particular, when $H$ is row--finite and faithful, we will give an inductive limit decomposition of the crossed product $\mathcal{O}_A(H) \rtimes_\gamma \mathbb{T}$.  Combining this with the continuity of $K$\=/theory, we will calculate the $K$\=/theory of this crossed product and hence prove Theorem A.

Suppose $H$ is a $C^*$\=/correspondence over a $C^*$-algebra $A$.  Define actions
\[ \alpha : \mathbb{T} \rightarrow \operatorname{Aut}(A) \qquad \text{and} \qquad \beta : \mathbb{T} \rightarrow \operatorname{Aut}(H) \]
by $\alpha_z(a) = a$ and $\alpha_z(\xi) = z \xi$.  It is easily verified that $(\alpha, \beta)$ is a correspondence action of $\mathbb{T}$ on $(A, H)$.  Our first goal is to give a concrete description of the correspondence $(A \rtimes_\alpha \mathbb{T}, H \rtimes_\beta \mathbb{T})$.

Let $A^\infty = C_0(\mathbb{Z}, A)$ and $H^\infty = C_0(\mathbb{Z}, A)$.  Then $H^\infty$ is a $C^*$\=/correspondence over $A^\infty$ with
\[ (a \xi)(n) = a(n - 1) \xi(n), \qquad (\xi a)(n) = \xi(n) a(n), \qquad \langle \xi, \eta \rangle(n) = \langle \xi(n), \eta(n) \rangle, \]
for every $a \in A^\infty$, $\xi, \eta \in H^\infty$, and $n \in \mathbb{Z}$.  Define
\[ f : C(\mathbb{T}, A) \rightarrow A^\infty \qquad \text{and} \qquad g : C(\mathbb{T}, H) \rightarrow H^\infty \]
by
\[ f(a)(n) = \int_\mathbb{T} z^{-n} a(z) \, dz \qquad \text{and} \qquad g(\xi)(n) = \int_\mathbb{T} z^{-n} \xi(z) \, dz \]
It is clear that $f$ extends to an isomorphism from $A \rtimes_\alpha \mathbb{Z}$ onto $A^\infty$.  Moreover, since
\begin{align*}
      \langle g(\xi), g(\eta) \rangle(n)
   &= \langle g(\xi)(n), g(\eta)(n) \rangle \\
   &= \left\langle \int_{\mathbb{T}} z^{-n} \xi(z) \, dz, \int_{\mathbb{T}} w^{-n} \eta(w) \, dw \right\rangle \\
   &= \int_{\mathbb{T}} \int_{\mathbb{T}} z^n w^{-n} \langle \xi(z), \eta(w) \rangle \, dw \, dz \\
   &= \int_{\mathbb{T}} \int_{\mathbb{T}} w^{-n} \langle \xi(z), \eta(zw) \rangle \, dw \, dz \\
   &= \int_{\mathbb{T}} w^{-n} \langle \xi, \eta \rangle(w) \, dw \\
   &= f(\langle \xi, \eta \rangle)(n)
\end{align*}
for every $\xi, \eta \in C(\mathbb{T}, H)$ and $n \in \mathbb{Z}$, the map $g$ is isometric and hence extends by continuity to an isometric map $g$ from $H \rtimes_\beta \mathbb{Z}$ onto $H^\infty$.  A routine calculation shows that $f(a) g(\xi) = g(a \xi)$ and $g(\xi) f(a) = g(\xi a)$ for every $a \in C(\mathbb{T}, A)$ and $\xi \in C(\mathbb{T}, H)$.  Hence the correspondences $(A^\infty, H^\infty)$ and $(A \rtimes_\alpha \mathbb{T}, H \rtimes_\beta \mathbb{T})$ are isomorphic.

It is clear from the construction above that the correspondence action $(\alpha, \beta)$ induces the usual gauge action $\gamma$ of $\mathbb{T}$ on $\mathcal{O}_A(H)$; that is, $\gamma(\pi(a)) = \pi(a)$ and $\gamma(\tau(\xi)) = z \tau(\xi)$ for every $a \in A$ and $\xi \in H$.  By Theorem \ref{thm:HaoNg} and by the calculation above, there are isomorphisms
\[ \mathcal{O}_{A^\infty}(H^\infty) \cong \mathcal{O}_{A \rtimes_\alpha \mathbb{T}}(H \rtimes_\beta \mathbb{T}) \cong \mathcal{O}_A(H) \rtimes_\gamma \mathbb{T}.\]
We will describe the automorphism of $\mathcal{O}_{A^\infty}(H^\infty)$ induced by the automorphism $\hat{\gamma}$ of $\mathcal{O}_A(H) \rtimes_\gamma \mathbb{T}$ through these isomorphisms.

Let $(\hat{\alpha}, \hat{\beta})$ be the automorphism of the correspondence $(A \rtimes_\alpha \mathbb{T}, H \rtimes_\alpha \mathbb{T})$ given by $\hat{\alpha}(a)(z) = z^{-1} a(z)$ and $\hat{\beta}(\xi)(z) = z^{-1} \xi(z)$ for every $a \in C(\mathbb{T}, A)$ and $\xi \in C(\mathbb{T}, H)$.
We denote by $\sigma$ the ``left shift'' operators $\sigma : A^\infty \rightarrow A^\infty$ and $\sigma : H^\infty \rightarrow H^\infty$ given by $\sigma(x)(n) = x(n + 1)$ for every $x \in A \cup H$.  These shift operators induce an automorphism $\sigma$ of $\mathcal{O}_{A^\infty}(H^\infty)$.  Note that for each $a \in C_c(G, A)$ and $s \in G$, we have
\[ f(\hat{\alpha}(a))(n) = \int_{\mathbb{T}} z^{-n} \hat{\alpha}(a)(z) \, dz = \int_{\mathbb{T}} z^{-n - 1} a(z) \, dz = f(a)(n + 1) = \sigma(f(a))(n). \]
Hence $f \circ \hat{\alpha} = \sigma \circ f$ and similarly  $g \circ \hat{\beta} = \sigma \circ g$.

Combining the above calculation with Proposition \ref{prop:equivariant}, we have proven the following result.

\begin{proposition}\label{prop:SkewProduct}
Suppose $A$ is a $C^*$-algebra, $H$ is a $C^*$\=/correspondence over $A$, and let $\gamma$ denote the standard gauge action of $\mathbb{T}$ on $\mathcal{O}_A(H)$.  Then there is an isomorphism
\[ \mathcal{O}_A(H) \rtimes_\gamma \mathbb{T} \cong \mathcal{O}_{A^\infty}(H^\infty) \]
that intertwines the automorphisms $\hat{\gamma}$ of $\mathcal{O}_A(H) \rtimes_\gamma \mathbb{T}$ and $\sigma$ of $\mathcal{O}_{A^\infty}(H^\infty)$.
\end{proposition}

We can use the above proposition to give an inductive limit decomposition of the crossed product $\mathcal{O}_A(H) \rtimes_\gamma \mathbb{T}$.  Given $n \in \mathbb{Z}$, define a subcorrespondence $(A^n, H^n)$ of $(A^\infty, H^\infty)$ by
\begin{align*}
   A^n &= \{ a \in A^\infty : a(k) = 0 \text{ for } k > n \}
\intertext{and}
   H^n &= \{ \xi \in H^\infty : \xi(k) = 0 \text{ for } k > n \}.
\end{align*}
Note that the automorphism $\sigma$ of $(A^\infty, H^\infty)$ restricts to an endomorphism $\sigma$ of $(A^n, H^n)$.  The following proposition is easily verified.

\begin{proposition}\label{prop:InductiveLimit}
For $n < m < \infty$, the inclusions
\[ A^n \hookrightarrow A^m \hookrightarrow A^\infty \qquad \text{and} \qquad H^n \hookrightarrow H^m \hookrightarrow H^\infty \] induce embeddings
\[ \begin{tikzcd} \mathcal{O}_{A^n}(H^n) \arrow[hook]{r}{\iota_{n, m}} & \mathcal{O}_{A^m}(H^m) \arrow[hook]{r}{\iota{m, \infty}} & \mathcal{O}_{A^\infty}(H^\infty) \end{tikzcd}. \]
Moreover, the maps $\iota_{m, \infty}$ induce an isomorphism
\[ \underset{\longrightarrow}{\lim} \, (\mathcal{O}_{A^n}(H^n), \iota_{n, m}) \rightarrow \mathcal{O}_{A^\infty}(H^\infty). \]
\end{proposition}

We will give a more explicit description of the algebras $\mathcal{O}_{A^n}(H^n)$ in the case where $H$ is row--finite and faithful.  In particular, there is a faithful representation of $\mathcal{O}_{A^n}(H^n)$ on the Fock space $\mathcal{F}_A(H)$.  Using this representation, we will show each $\mathcal{O}_{A^n}(H^n)$ is Morita equivalent to $A$.  This Morita equivalence will play a key role in the proof of Theorem A.

\begin{proposition}\label{prop:FockBimodule}
Suppose $H$ is a regular $C^*$\=/correspondence over a $C^*$\=/algebra $A$.  There is a canonical isomorphism $\mathcal{O}_{A^n}(H^n) \rightarrow \mathbb{K}(\mathcal{F}_A(H))$.  In particular, $\mathcal{F}_A(H)$ is an $\mathcal{O}_{A^n}(H^n)$--$A$ imprimitivity bimodule and $\mathcal{O}_{A^n}(H^n)$ is Morita equivalent to $A$.
\end{proposition}

\begin{proof}
Let $H$ denote a row--finite, faithful $C^*$\=/correspondence over $A$.  For $n \in \mathbb{Z}$ and $k \geq 0$, define $\pi^n_k : A^n \rightarrow \mathbb{B}(H^{\otimes k})$ and $\tau^n_k : H^n \rightarrow \mathbb{B}(H^{\otimes k}, H^{\otimes (k+1)})$ by
\begin{align*}
   \pi^n_k(a)(\eta_1 \otimes \cdots \otimes \eta_k) &= a(n - k) \eta_1 \otimes \cdots \otimes \eta_k \\
\intertext{and}
   \tau^n_k(\xi)(\eta_1 \otimes \cdots \otimes \eta_k) &= \xi(n - k) \otimes \eta_1 \otimes \cdots \otimes \eta_k.
\end{align*}

We identify $H^{\otimes k}$ with $H \otimes_A H^{\otimes (k - 1)}$, $\pi^n_k(a) = \lambda(a(n-k)) \otimes \operatorname{id}$.  Since $H$ is row--finite, $\lambda(a(n-k))$ is compact and hence $\pi^n_k(a)$ is compact. If $e_i$ is an approximate identity in $A$, let $a_i \in A^n$ be the element given by $a_i(n - k) = e_i$ and $a_i(n - \ell) = 0$ for $\ell \neq k$.  Then $\tau^n_k(\xi) \pi^n_k(a_i) = \tau^n_k(\xi e_i) \rightarrow \tau^n_k(\xi)$.  Since each $\pi^n_k(a_i)$ is compact, $\tau^n_k(\xi)$ is compact.

Since $H$ is faithful, each $H^{\otimes (k - 1)}$ is faithful, and
\[ \|\pi^n_k(a)\| = \| \lambda(a(n - k)) \otimes \operatorname{id} \| = \| \lambda(a(n-k)) \| = \|a(n - k)\| \]
Moreover, for every $\eta \in H^{\otimes k}$,
\begin{align*}
     \| \tau^n_k(\xi)(\eta) \|^2
  &= \| \langle \xi(n -k) \otimes \eta, \xi(n-k) \otimes \eta \rangle \| \\
  &= \| \langle \eta, \langle \xi(n-k), \xi(n-k) \rangle \eta \rangle \| \\
  &= \| \langle \langle \xi(n-k), \xi(n-k) \rangle^{1/2} \eta, \langle \xi(n-k), \xi(n-k) \rangle^{1/2} \eta \rangle \| \\
  &= \| \langle \xi, \xi \rangle(n-k)^{1/2} \eta \|^2 \\
  &= \| \pi^n_k(\langle \xi, \xi \rangle)^{1/2} (\eta) \|^2
\end{align*}
Now,
\[ \| \tau^n_k(\xi) \| = \| \pi^n_k(\langle \xi, \xi \rangle) \|^{1/2} = \| \langle \xi, \xi \rangle(n -k) \|^{1/2} = \|\xi(n-k)\|. \]
Since $\| \pi^n_k(a) \| \rightarrow 0$ and $\| \tau^n_k(\xi) \| \rightarrow 0$ as $k \rightarrow \infty$ for each $a \in A^n$ and $\xi \in H^n$, we may define
\[ \pi^n(a) = \bigoplus_k \pi^n_k(a) \in \mathbb{K}(\mathcal{F}_A(H)) \qquad \text{and} \qquad \tau^n(\xi) = \bigoplus_k \tau^n_k(\xi) \in \mathbb{K}(\mathcal{F}_A(H)) \]
for all $a \in A^n$ and $\xi \in H^n$.

We claim $(\pi^n, \tau^n)$ is a covariant Toeplitz representation of $(A^n, H^n)$ on the $C^*$\=/algebra $\mathbb{K}(\mathcal{F}_A(H))$.  It is clear that $\pi^n$ is a *\=/homomorphism and $\tau^n$ is linear.  Note that if $\xi \in H^n$, then
\[ \tau^n(\xi)^*(\eta_1 \otimes \cdots \otimes \eta_k) = \langle \xi(n - k + 1), \eta_1 \rangle \eta_2 \otimes \cdots \otimes \eta_k \]
and $\tau^n(\xi)^*(a) = 0$ for every $a \in A$.  Now routine calculations show that
\[ \tau^n(\xi)^* \tau^n(\eta) = \pi^n(\langle \xi, \eta \rangle) \]
and
\[ \pi^n(a) \tau^n(\xi) = \tau^n(a \xi), \]
and hence $(\pi^n, \tau^n)$ is a Toeplitz representation.

It remains to show the representation $(\pi^n, \tau^n)$ is covariant.  Let $J^n$ denote the Katsura ideal of $(A^n, H^n)$.  Since $H$ is row-finite, so is $H^n$.  Thus $J^n = \ker( \lambda^n )^\perp$, where $\lambda^n$ is the left multiplication operator $A^n \rightarrow \mathbb{B}(H^n)$.  We claim
\[ \ker \lambda^n = \{ a \in A^n : a(n - k) = 0 \text{ for every } k \geq 1 \}. \]
It is clear that if $a(n - k) = 0$ for every $k \geq 1$, then $\lambda^n(a) = 0$.  Suppose $k \geq 1$ such that $a(n - k) \neq 0$.  Since $H$ is faithful, there is a $\xi \in H$ such that $a(n - k) \xi \neq 0$.  Choose $\xi' \in H^\infty$ such that $\xi'(n - k + 1) = \xi$ and note that
\[ (a \xi')(n - k + 1) = a(n - k) \xi'(n - k + 1) = a(n - k) \xi \neq 0. \]
Thus $\lambda^n(a) \neq 0$ and the claim holds.  Thus $J^n = A^{n-1} \subseteq A^n$.

Let $\varphi^n: \mathbb{K}(H^n) \rightarrow \mathbb{K}(\mathcal{F}_A(H))$ be given by $\varphi^n(\theta_{\xi, \eta}) = \tau^n(\xi) \tau^n(\eta)^*$.  Then note that
\begin{align*}
     \varphi^n(\theta_{\xi, \eta})(\zeta_1 \otimes \cdots \otimes \zeta_k)
  &= \tau^n(\xi) \tau^n(\eta)^*(\zeta_1 \otimes \cdots \otimes \zeta_k) \\
  &= \tau^n(\xi) (\langle \eta(n - k + 1), \zeta_1 \rangle \zeta_2 \otimes \zeta_k) \\
  &= \xi(n - k + 1) \otimes \langle \eta(n - k + 1), \zeta_1 \rangle \zeta_2 \otimes \zeta_k \\
  &= \xi(n - k + 1) \langle \eta(n - k + 1), \zeta_1 \rangle \otimes \zeta_2 \otimes \cdots \otimes \zeta_k \\
  &= \theta_{\xi(n - k +1), \eta(n - k + 1)}(\zeta_1) \otimes \zeta_2 \otimes \cdots \otimes \zeta_k.
\end{align*}
Now it is easy to verify $\varphi^n(\lambda(a)) = \pi^n(a)$ whenever $a \in J^n$.

Now $(\pi^n, \tau^n)$ is a covariant Toeplitz representation, and the induced *\=/homomorphism $\psi^n : \mathcal{O}_{A^n}(H^n) \rightarrow \mathbb{K}(\mathcal{F}_A(H))$ is injective by the gauge invariance uniqueness theorem.  It only remains to show $\psi^n$ is surjective.  However, this is routine.  For instance, if $\xi = \xi_1 \otimes \cdots \otimes \xi_k$ and $\eta = \eta_1 \otimes \cdots \otimes \eta_\ell$, then
\[ \theta_{\xi, \eta} = \tau^n(j_{n-k+1}(\xi_1)) \cdots \tau^n(j_n(\xi_k)) \tau^n(j_n(\eta_\ell))^* \cdots \tau^n(j_{n - \ell + 1}(\eta_1))^*,  \]
where the $j_i: H \rightarrow H^n$ are the inclusion maps into the $i$ coordinate.
\end{proof}

Using the previous result together with the continuity of $K$\=/theory and the Fredholm picture of $K$\=/theory outlined in Section \ref{sec:KTheory}, we can now prove Theorem A.

\begin{proof}[Proof of Theorem A]
We claim the diagram
\begin{center}
\begin{tikzcd}[row sep = large]
& A \arrow{rr}{H} \arrow{dd}[near start]{H} & & A \arrow{dd}{H} \\
\mathcal{O}_{A^n}(H^n) \arrow[crossing over]{rr}[near start]{\iota_n} \arrow{dd}[swap]{\sigma} \arrow{ur}[rotate = 37, pos = .8]{\mathcal{F}_A(H)} & & \mathcal{O}_{A^{n+1}}(H^{n+1}) \arrow{ur}[rotate = 35, pos = .8]{\mathcal{F}_A(H)} & \\
& A \arrow{rr}[near start]{H} & & A \\
\mathcal{O}_{A^n}(H^n) \arrow{ur}[rotate = 37, pos = .7]{\mathcal{F}_A(H)} \arrow{rr}{\iota_n} & & \mathcal{O}_{A^{n+1}}(H^{n+1}) \arrow{ur}[rotate = 35, pos = .7]{\mathcal{F}_A(H)} \arrow[crossing over, leftarrow]{uu}[near end]{\sigma} &
\end{tikzcd}
\end{center}
commutes on $K$-theory.  It is clear that the front and back faces of the cube commute on $K$-theory.  We will show the top and bottom faces commute.  The argument for the left and right faces is similar.

Let $L$ denote the tensor product across the top and right arrows; that is,
\[ L = \mathcal{O}_A^{n+1}(H^{n+1}) \otimes_{\psi^n} \mathcal{F}_A(H) \cong \mathcal{F}_A(H). \]
We will identify $L$ with $\mathcal{F}_A(H)$ and we note that $L$ is viewed a $\mathcal{O}_{A^n}(H^n)$--$A$ $C^*$\=/correspondence with the left action given by $x \cdot \xi = \iota_n(x)(\xi)$ for every $x \in \mathcal{O}_{A^n}(H^n)$ and $\xi \in L$.  Similarly, we define $K$ to be the correspondence $K = \mathcal{F}_A(H) \otimes_A H$.  To prove the claim we must show $[K] = [L]$.

Define $f : K \rightarrow L$ by $f(\xi \otimes \eta) = \xi \otimes \eta$ for every $\xi \in \mathcal{F}_A(H)$ and $\eta \in H$.  It is clear that $f$ is isometric and preserves the right module structure.  A tedious computation shows that $f$ preserves the left module structure.  For example,
\begin{align*}
   \pi^n(a) \cdot f(\zeta \otimes \zeta') &= \iota_n(\pi^n(a))(\zeta \otimes \zeta') = \pi^{n+1}(a)(\zeta \otimes \zeta') = a(n - k) \zeta \otimes \zeta' \\ &= (\pi^n(a) \zeta) \otimes \zeta' = f( \pi^n(a) \cdot (\zeta \otimes \zeta')).
\end{align*}
for $\zeta \in H^{\otimes k}$ and $\zeta' \in H$.

We have shown that $f$ is a correspondence isomorphism onto $f(K) \subseteq L$.  Therefore, $[K] = [f(K)]$.  To see that $[f(K)] = [L]$, it suffices to show $f(K) \supseteq \mathcal{O}_{A^n}(H^n) \cdot L$  by Proposition \ref{prop:NonDegKTheory}.  If $b \in A \subseteq L$, then for every $a \in A^n$ and $\xi \in H^n$, we have
\begin{align*}
  \pi^n(a) \cdot b &= \iota_n(\pi^n(a))(b) = \pi^{n+1}(a)(b) = a(n+1)(b) = 0, \\
  \tau^n(\xi) \cdot b &= \iota_n(\tau^n(\xi)) \tau^{n+1}(\xi)(b) = \xi(n+1)b = 0, \text{ and} \\
  \tau^n(\xi)^* \cdot b &= \iota_n(\tau^n(\xi))^*(b) = \tau^{n+1}(\xi)^*(b) = 0.
\end{align*}
It follows that
\[ \mathcal{O}_{A^n}(H^n) \cdot L \subseteq \bigoplus_{k=1}^\infty H^{\otimes k} \subseteq f(K). \]
This proves the claim.

By Propositions \ref{prop:InductiveLimit} and \ref{prop:SkewProduct},
\[ \mathcal{O}_A(H) \rtimes_\gamma \mathbb{T} \cong \underset{\longrightarrow}{\lim} \, (\mathcal{O}_{A^n}(H^n), \iota_n) \]
where $\iota_n : \mathcal{O}_{A^n}(H^n) \rightarrow \mathcal{O}_{A^{n+1}}(H^{n+1})$ is the map induced by the inclusions $A^n \hookrightarrow A^{n+1}$ and $H^n \hookrightarrow H^{n+1}$.  Moreover, the isomorphism intertwines the automorphisms $\hat{\gamma}$ of $\mathcal{O}_A(H) \rtimes_\gamma \mathbb{T}$ and the endomorphisms $\sigma$ of $\mathcal{O}_{A^n}(H^n)$ by Proposition \ref{prop:SkewProduct}.

By Proposition \ref{prop:CorrKTheory}, $H$ induces a group homomorphism $[H] : K_*(A) \rightarrow K_*(A)$ that is positive on $K_0$ groups.  Proposition \ref{prop:FockBimodule} shows $\mathcal{F}_A(H)$ is an $\mathcal{O}_{A^n}(H^n)$--$A$ imprimitivity bimodule.  Hence the map
\[ [\mathcal{F}_A(H)] : K_*(\mathcal{O}_{A^n}(H^n)) \rightarrow K_*(A) \]
is an isomorphism (see Theorem 5.3 in \cite{ExelFredholmOperators}).  Now the result follows from the continuity of $K$-theory and the commutativity of the diagram above.
\end{proof}

\begin{remark}
If $A$ is AF and $H$ is a row--finite, faithful $C^*$\=/correspondence over $A$, then Theorem B implies $\mathcal{O}_A(H) \rtimes_\gamma \mathbb{T}$ is AF.  Theorem A allows us to compute the ordered $K_0$ group of $\mathcal{O}_A(H) \rtimes_\gamma \mathbb{T}$ and hence, at least in principle, it is possible to determine $\mathcal{O}_A(H) \rtimes_\gamma \mathbb{T}$ up to Morita equivalence.
\end{remark}

The following example shows the utility of these results by giving a proof of the well--known fact that $\mathcal{O}_n \rtimes_\gamma \mathbb{T}$ is Morita equivalent to the UHF algebra of type $n^\infty$.

\begin{example}
Let $A = \mathbb{C}$ and $H = \mathbb{C}^n$ with $2 \leq n < \infty$.  Then $H$ is a $C^*$\=/correspondence over $A$ with the obvious operations and $\mathcal{O}_A(H)$ is the usual Cuntz algebra $\mathcal{O}_n$.  By Theorem B, $\mathcal{O}_n \rtimes_\gamma \mathbb{T}$ is an AF\=/algebra.  Identifying $K_0(\mathbb{C})$ with $\mathbb{Z}$, the map $[H] : \mathbb{Z} \rightarrow \mathbb{Z}$ is given by multiplication by $n$.  Hence by Theorem A,
\[ K_0(\mathcal{O}_A(H) \rtimes_\gamma \mathbb{T}) \cong \left\{ \frac{a}{n^k} : a \in \mathbb{Z}, \, k \geq 0 \right\} \subseteq \mathbb{Q} \]
with the usual order structure.  Hence $\mathcal{O}_A(H) \rtimes_\gamma \mathbb{T}$ is Morita equivalent to the UHF algebra of type $n^\infty$.
\end{example}

\section{The Approximately Finite--Dimensional Case}\label{sec:AFAlgebras}

Our goal is to prove Theorem B.  That is, if $A$ is AF, then so is $\mathcal{O}_A(H) \rtimes_\gamma \mathbb{T}$.  First consider a separable, row--finite, faithful $C^*$\=/correspondence $H$ over an AF\=/algebra $A$.  Then by Proposition \ref{prop:FockBimodule}, $\mathcal{O}_{A^n}(H^n)$ is Morita equivalent to $A$ and hence is AF.  Since inductive limits of AF\=/algebras are AF, Proposition \ref{prop:InductiveLimit} implies $\mathcal{O}_A(H) \rtimes_\gamma \mathbb{T}$ is also AF.

The general case will require more work.  We have already seen in Proposition \ref{prop:SkewProduct} that $\mathcal{O}_A(H) \rtimes_\gamma \mathbb{T} \cong \mathcal{O}_{A^\infty}(H^\infty)$.  We will show $\mathcal{O}_{A^\infty}(H^\infty)$ is AF by giving an alternate description of this algebra.  Some of our arguments will require $H$ to be non--degenerate.  Let us first show we can reduce to this case.

\begin{proposition}
The algebras $\mathcal{O}_A(H) \rtimes_\gamma \mathbb{T}$ and $\mathcal{O}_A(AH) \rtimes_\gamma \mathbb{T}$ are Morita equivalent.
\end{proposition}

\begin{proof}
First note that if $a \in A^\infty$ and $\xi \in H^\infty$.  Then $a \xi \in (AH)^\infty$.  Hence $A^\infty H^\infty \subseteq (AH)^\infty$ as a subcorrespondence.  Suppose $\xi \in (AH)^\infty$ with finite support.  For each $n \in \mathbb{Z}$, we may choose $a(n - 1) \in A$ and $\eta(n) \in H$ such that $a(n - 1) \eta(n) = \xi(n)$ by the Cohen Factorization Theorem.  Moreover, by choosing $a(n - 1) = 0$ and $\eta(n) = 0$ whenever $\xi(n) = 0$, we have $a \in A^\infty$, $\xi \in H^\infty$, and $a \eta = \xi$.  Hence $A^\infty H^\infty = (AH)^\infty$.  Now by Propositions \ref{prop:NonDegenerate} and \ref{prop:SkewProduct}, we have
\[ \mathcal{O}_A(H) \rtimes_\gamma \mathbb{T} \cong \mathcal{O}_{A^\infty}(H^\infty) \sim \mathcal{O}_{A^\infty}(A^\infty H^\infty) = \mathcal{O}_{A^\infty}((AH)^\infty) \cong \mathcal{O}_A(AH) \rtimes_\gamma \mathbb{T}. \]
\end{proof}

Consider the maps $j_n = j_n^A : A \rightarrow A^\infty$ and $j_n = j_n^H : H \rightarrow H^\infty$ given by the inclusion into the $n$--th coordinate. The following lemma is a simple computation.

\begin{lemma}
Suppose $a \in A$, $\xi, \eta \in H$, and $k, l \in \mathbb{Z}$ are given.
\begin{enumerate}
  \item $\displaystyle j_k^A(a) j_l^H(\xi) = \begin{cases} j_l^H(a \xi) & k = l + 1 \\ 0 & else \end{cases}$
  \item $\displaystyle j_k^H(\xi) j_l^A(a) = \begin{cases} j_l^H(\xi a) & k = l \\ 0 & else \end{cases}$
  \item $\displaystyle \langle j_k^H(\xi), j_l^H(\eta) \rangle = \begin{cases} j_l^A(\langle \xi, \eta \rangle) & k = l \\ 0 & else \end{cases}$.
\end{enumerate}
\end{lemma}

Now suppose $H$ is a non\=/degenerate $C^*$\=/correspondence over $A$.  It is easily verified that the canonical map $A^\infty \hookrightarrow \mathcal{O}_A{^\infty}(H^\infty)$ is also non\=/degenerate.  Hence if $j_n : A \hookrightarrow A^{\infty}$ extends to a *\=/homomorphism $j_n : M(A) \rightarrow M(\mathcal{O}_{A^\infty}(H^\infty))$.  It is easy to see that the $j_n(1)$ are mutually orthogonal projections and $\sum_{n \in \mathbb{Z}} j_n(1) = 1$ (strict topology).

The following proposition is essentially proven in \cite[Propostion 4.1]{EvansSims}.  For the reader's convenience, we give a complete proof.

\begin{proposition}\label{lem:AFCorners}
Suppose $A$ is an $C^*$\=/algebra and $(p_i)_{i \in I} \subseteq M(B)$ is a countable collection of pairwise orthogonal projections such that $p_i B p_i$ is AF for every integer $i$ and $\sum p_i = 1$ (strict convergence).  Then $B$ is AF.
\end{proposition}

\begin{proof}
Suppose first that $p, q \in M(B)$ are projections such that $p + q = 1$ and $pBp$ and $qBq$ are AF. Then $BpB$ and $BqB$ are AF since they are, respectively, Morita equivalent to $pBp$ and $qBq$.  We claim $BpB + BqB = B$.  Choose an approximate identity $(e_n) \subseteq B$.  For each $a \in B$,
\[ \lim \, (e_n p a + e_n q a) = p a + q a = (p + q) a = a. \]
So $a \in BpB + BqB$ and the claim holds.  Consider the exact sequence
\[ \begin{tikzcd} 0 \arrow{r} & B p B \arrow{r} & B \arrow{r} & (BqB)/(BpB \cap BqB) \arrow{r} & 0. \end{tikzcd} \]
Since the class of AF\=/algebras are closed under taking quotients and extensions, $A$ is AF.

For convenience, we assume the index set $I$ is the natural numbers.  Now set $q_n = p_1 + \cdots + p_n$.  Since $q_n \rightarrow 1$ strictly in $M(A)$, we have $q_n a q_n \rightarrow a$ in $B$, and hence $B = \underset{\longrightarrow}{\lim} \, q_n B q_n$.  Applying the argument in the previous paragraph by induction, $q_n B q_n$ is AF for each $n$.  As inductive limits of AF\=/algebras are AF, the result follows.
\end{proof}

By the previous proposition, to show $\mathcal{O}_{A^\infty}(H^\infty)$ is AF, it is enough to show each $j_n(1) \mathcal{O}_{A^\infty}(H^\infty)j_n(1)$ is AF for each $n \in \mathbb{Z}$.  To do this, we will first show the corner $j_n(1) \mathcal{O}_{A^\infty}(H^\infty)j_n(1)$ is isomorphic to the fixed point algebra $\mathcal{O}_A(H)^\gamma$ by building an embedding
\[ \psi_n : \mathcal{O}_A(H) \rightarrow \mathcal{O}_A(H) \rtimes_\sigma \mathbb{Z} \]
such that $\psi_n(\mathcal{O}_A(H)^\gamma) = j_n(1) \mathcal{O}_{A^\infty}(H^\infty)j_n(1)$.

Let $(\pi, \tau) : (A, H) \rightarrow \mathcal{O}_A(H)$ and $(\pi^\infty, \tau^\infty) : (A^\infty, H^\infty) \rightarrow \mathcal{O}_{A^\infty}(H^\infty)$ be the universal covariant representations.  Let
\[ \varphi : \mathbb{K}(H) \rightarrow \mathcal{O}_A(H) \qquad \text{and} \qquad \varphi^\infty : \mathbb{K}^\infty(H^\infty) \rightarrow \mathcal{O}_{A^\infty}(H^\infty) \]
denote the *\=/homomorphisms defined in Proposition \ref{prop:InducedMapOnCompacts}.

If $\sigma$ is the automorphism of $\mathcal{O}_{A^\infty}(H^\infty)$ given by the left bilateral shift, then
\[ \sigma(\pi^\infty(j_n(a))) = \pi^\infty(j_{n-1}(a)) \qquad \text{and} \qquad  \sigma(\tau^\infty(j_n(\xi))) = \tau^\infty(j_{n-1}(\xi)) \]
for $a \in A$, $\xi \in H$, and $n \in \mathbb{Z}$.  Let $B = \mathcal{O}_{A^\infty}(H^\infty) \rtimes_\sigma \mathbb{Z}$ and let $u \in M(B)$ be a unitary such that $uxu^* = \sigma(x)$ for every $x \in \mathcal{O}_{A^\infty}(H^\infty)$.

Fix $n \in \mathbb{Z}$.  Define $\pi_n: A \rightarrow B$ and $\tau_n : H \rightarrow B$ by $\pi_n(a) = \pi^\infty(j_n(a))$ and $\tau_n(a) = u \tau^\infty(j_n(a))$.

\begin{lemma}
The pair $(\pi_n, \tau_n)$ is a Toeplitz representation.
\end{lemma}

\begin{proof}
First note that given $\xi, \eta \in H$ and $a \in A$, we have
\begin{align*}
\tau_n(\xi)^* \tau_n(\eta) &= \tau^\infty(j_n(\xi))^*\tau^\infty(j_n(\eta)) = \pi^\infty(\langle j_n(\xi), j_n(\eta) \rangle) \\
 &= \pi^\infty( j_n(\langle \xi, \eta \rangle)) = \pi_n(\langle \xi, \eta \rangle),
\end{align*}
and
\begin{align*}
\pi_n(a) \tau_n(\xi) &= \pi^\infty(j_n(a)) u \tau^\infty(j_n(\xi)) = u \pi^\infty(j_{n+1}(a)) \tau^\infty(j_n(\xi)) \\
&= u \tau^\infty(j_{n+1}(a) j_n(\xi)) = u \tau^\infty(j_n(a \xi)) = \tau_n(a\xi).
\end{align*}
Hence $(\pi_n, \tau_n)$ is a Toeplitz representation.
\end{proof}

We claim the Toeplitz representation $(\pi_n, \tau_n)$ is covariant.  Let $\varphi_n : \mathbb{K}(H) \rightarrow B$ be the *\=/homomorphism defined in Proposition \ref{prop:InducedMapOnCompacts}.  We must show $\varphi_n(\lambda(a)) = \pi_n(a)$ for every $a \in J_H$.  This will require a preliminary results.

\begin{lemma}
Fix $n \in \mathbb{Z}$.
\begin{enumerate}
  \item If $T \in \mathbb{K}(H)$, then $j_n T j_n^* \in \mathbb{K}(H^\infty)$.
  \item If $a \in A$, then $\lambda^\infty(j_n(a)) = j_{n-1} \lambda(a) j_{n-1}^*$.
  \item If $a \in J_H$, then $j_n(a) \in J_{H^\infty}$.
  \item The following diagram commutes:
    \begin{center}
    \begin{tikzcd}[row sep = small]
       J_H \arrow{r}{\lambda} \arrow{dd}[swap]{j_n} & \mathbb{K}(H) \arrow{dr}{\varphi_n} \arrow{dd}[swap]{\operatorname{ad}(j_{n-1})}& \\
       & & B \\
       J_{H^\infty} \arrow{r}{\lambda^\infty} & \mathbb{K}(H^\infty) \arrow{ur}[swap]{\varphi^\infty} &
    \end{tikzcd}
    \end{center}
\end{enumerate}
\end{lemma}

\begin{proof}
If $T = \theta_{\xi, \eta}$, we have $j_n \theta_{\xi, \eta} j_n^* = \theta_{j_n(\xi), j_n(\eta)^*} \in \mathbb{K}(H^\infty)$.  Now part (1) follows from the linearity and continuity of $j_n$ and $j_n^*$.  Part (2) follows from a simple computation.  We prove part (3).  If $\lambda(a) \in \mathbb{K}(H)$, then $\lambda^\infty(j_n(a)) = j_{n-1} \lambda(a) j_{n-1}^* \in \mathbb{K}(H^\infty)$.  Now suppose $a \in (\ker \lambda)^\perp$ and $b \in \ker(\lambda^\infty)$.  For every $\xi \in H$,
\begin{align*}
\lambda(b(n))(\xi) &= b(n) \xi = b(n) j_{n-1}(\xi)(n-1) = (b j_{n-1}(\xi))(n-1) \\
&= \lambda^\infty(b)(j_{n-1}(\xi))(n-1) = 0.
\end{align*}
Hence $b(n) \in \ker \lambda$ and $a b(n) = 0$.  That is, $(j_n(a) b)(n) = 0$.  Clearly $(j_n(a) b)(k) = 0$ for $k \neq n$, and hence $j_n(a)b = 0$.  That is, $j_n(a) \in (\ker \lambda^\infty)^\perp$.  This proves part (3).  Part (4) is easily verified.
\end{proof}

\begin{proposition}
The pair $(\pi_n, \tau_n)$ is a covariant Toeplitz representation and induces an embedding $\psi_n : \mathcal{O}_A(H) \hookrightarrow B$.
\end{proposition}

\begin{proof}
If $a \in J_H$, then $j_n(a) \in J_{H^\infty}$.  Since the Toeplitz representation $(\pi^\infty, \tau^\infty)$ is covariant, we have
\[ \pi_n(a) = \pi^\infty(j_n(a)) = \varphi^\infty(\lambda^\infty(j_n(a))) = \varphi_n(\lambda(a)) \]
That is, $(\pi_n, \tau_n)$ is covariant.  Let $\psi_n : \mathcal{O}_A(H) \rightarrow B$ be the *\=/homomorphism defined by $\psi_n(\pi(a)) = \pi_n(a)$ and $\psi_n(\tau(\xi)) = \tau_n(\xi)$.  By the Gauge Invariant Uniqueness Theorem, $\psi_n$ is injective.
\end{proof}

We will now wish to calculate the range of $\psi_n$

\begin{proposition}\label{prop:RangeOfPsi}
For every $n \in \mathbb{Z}$, we have $\psi_n(\mathcal{O}_A(H)) = j_n(1) B j_n(1)$ and $\psi_n(\mathcal{O}_A(H)^\gamma) = j_n(1) \mathcal{O}_{A^\infty}(H^\infty) j_n(1)$.  In particular,
\[ \mathcal{O}_A(H)^\gamma \cong j_n(1) \mathcal{O}_{A^\infty}(H^\infty) \]
for every $n \in \mathbb{Z}$.
\end{proposition}

\begin{proof}
Note that
\begin{align*}
 B = \overline{\operatorname{span}} &\, \left\{ \tau^\infty(\xi_1) \cdots \tau^\infty(\xi_k) \pi^\infty(a) \tau^\infty(\eta_l)^* \cdots \tau^\infty(\eta_1)^* u^m : \right. \\ &\, \left. \eta_i, \xi_i \in H^\infty, a \in A^\infty, k, l, m \in \mathbb{Z} \right\}.
\end{align*}
Since for every $\xi \in H^\infty$ and $a \in A^\infty$,
\[ \tau^\infty(\xi) = \sum_{k = - \infty}^\infty j_k(\xi(k)) \qquad \text{and} \qquad \pi^\infty(a) = \sum_{k = - \infty}^\infty j_k(a(k)), \]
it follows that
\begin{align*}
  B = \overline{\operatorname{span}} &\, \left\{ \tau_{i(1)}(\xi_1) \cdots \tau_{i(k)}(\xi_k) \pi_j(a) \tau_{i'(1)}(\eta_l)^* \cdots \tau_{i'(l)}(\eta_l)^* u^m : \right. \\ &\, \left. \xi_r, \eta_s \in H, a \in A, i(r), i'(s) \in \mathbb{Z}, 1 \leq r \leq k, 1 \leq s \leq l, \text{ and } j, k, l, m \in \mathbb{Z} \right\}.
\end{align*}
Now suppose $k, l \in \mathbb{Z}$ and $\xi, \eta \in H$.  Then
\[ \tau_k(\xi) \tau_l(\eta) = \tau_k(\xi) p_k p_{l - 1} \tau_l(\eta) = \begin{cases} \tau_{l-1}(\xi) \tau_l(\eta) & k = l - 1 \\ 0 & \text{else}. \end{cases} \]
Similarly,
\[ \tau_k(\xi) \pi_l(a) = \begin{cases} \tau_l(\xi a) & k = l \\ 0 & \text{else}. \end{cases} \qquad \text{and} \qquad
   \pi_k(a) \tau_l(\xi) = \begin{cases} \tau_{l-1}(a \xi) & k = l - 1 \\ 0 & \text{else}. \end{cases} \]
Now we have
\[ j_n(1)Bj_n(1) =\overline{\operatorname{span}}\,\{\tau_{n+1}(\xi_{n+1}) \cdots \tau_k(\xi_k) \pi_k(a) \tau_k(\eta_k)^* \cdots \tau_{m+1}(\eta_{m+1})^* u^{n-m} \}.\]
The proposition is now easily verified.  In particular, the last statement follows since $\mathcal{O}_A(H)^\gamma$ is the closed span of
\[ \{ \tau(\xi_n) \cdots \tau(\xi_1) \pi(a) \tau(\eta_1)^* \cdots \tau(\eta_n)^* : \xi_i, \eta_i \in H, a \in A, \, n \in \mathbb{Z} \}. \]
\end{proof}

By the previous proposition, it only remains to show $\mathcal{O}_A(H)^\gamma$ is AF whenever $A$ is AF.  To show this we recall a construction of Katsura which gives an alternate description of $\mathcal{O}_A(H)^\gamma$.

\begin{proposition}
There is a unique injective *\=/homomorphism $\varphi^n : \mathbb{K}(H^{\otimes n}) \rightarrow \mathcal{O}_A(H)$ such that if
\[ \xi = \xi_1 \otimes \cdots \otimes \xi_n \qquad \text{and} \qquad \eta = \eta_1 \otimes \cdots \eta_n \]
with $\xi_i, \eta_i \in H$, then
\[ \varphi^n(\theta_{\xi, \eta}) = \tau(\xi_1) \cdots \tau(\xi_n) \tau(\eta_n)^* \cdots \tau(\eta_1)^*. \]
In particular, $\varphi^0 = \pi$ and $\varphi^1= \varphi$.
\end{proposition}

Set $A_n = \varphi^n(\mathbb{K}(H^{\otimes n})) \subseteq \mathcal{O}_A(H)$ for each $n \geq 0$.  In particular, note that $A_n$ is Morita equivalent to the ideal $\langle H^{\otimes n}, H^{\otimes n} \rangle$ of $A$.  Set $A_{[0, n]} = A_0 + \cdots + A_n$.

\begin{lemma}
For each $n \geq 0$, $A_{[0, n]}$ is a $C^*$\=/algebra and $A_n$ is an ideal of $A_{[0, n]}$.  Moreover, there is a short exact sequence
\[ \begin{tikzcd} 0 \arrow{r} & A_n \arrow{r} & A_{[0, n]} \arrow{r} & A_{[0, n-1]} / (A_n \cap A_{[0, n-1]}) \arrow{r} & 0. \end{tikzcd} \]
\end{lemma}

\begin{theorem}
There is an isomorphism $\displaystyle \mathcal{O}_A(H)^\gamma \cong A_{[0, \infty]} := \underset{\longrightarrow}{\lim} \, A_{[0, n]}$.
\end{theorem}

Recall that the class of AF\=/algebras is closed under Morita equivalences, ideals, quotients, extensions, and direct limits.  The following corollary is an immediate consequence of these facts together with the previous theorem and the construction of $A_{[0, \infty]}$.

\begin{corollary}\label{cor:AFCore}
If $A$ is AF and $H$ is separable, then $\mathcal{O}_A(H)^\gamma$ is AF.
\end{corollary}

This completes the proof of Theorem B.

\section{Applications to AF--Embeddability}\label{sec:BrownsTheorem}

In \cite{BrownAFE}, Brown considered the finiteness of algebras of the form $A \rtimes \mathbb{Z}$ where $A$ is an AF\=/algebra.  In particular, Brown proves the following theorem.

\begin{theorem}[Brown's Embedding Thoerem]\label{thm:BrownsEmbedding}
If $A$ is an AF\=/algebra and $\alpha \in \operatorname{Aut}(A)$, then the following are equivalent:
\begin{enumerate}
  \item $A \rtimes_\alpha \mathbb{Z}$ is AF\=/embeddable;
  \item $A \rtimes_\alpha \mathbb{Z}$ is quasidiagonal;
  \item $A \rtimes_\alpha \mathbb{Z}$ is stably finite;
  \item if $x \in K_0(A)$ and $[\alpha](x) \leq x$, then $[\alpha](x) = x$.
\end{enumerate}
\end{theorem}

Using Brown's Embedding Theorem together with Theorems A and B, we can prove Theorem C.

\begin{proof}[Proof of Theorem C]
By Theorem B, $\mathcal{O}_A(H) \rtimes_\gamma \mathbb{T}$ is AF and by Takai duality, $\mathcal{O}_A(H)$ is Morita equivalent to $(\mathcal{O}_A(H) \rtimes_\gamma \mathbb{T}) \rtimes_{\hat{\gamma}} \mathbb{Z}$.  Hence the equivalence of the first three conditions follows from Brown's Embedding Theorem.

Now suppose $H$ is row--finite and faithful.  By Theorem A, there is an isomorphism
\[ K_0(\mathcal{O}_A(H) \rtimes_\gamma \mathbb{T}) \cong \underset{\longrightarrow}{\lim} \, (K_0(A), [H]) \]
of ordered groups that intertwines the automorphism $[\hat{\gamma}]$ of $K_0(\mathcal{O}_A(H) \rtimes_\gamma \mathbb{T})$ and the automorphism of $\underset{\longrightarrow}{\lim} \, (K_0(A), [H])$ induced by $[H]$.  Let $j$ be the composition
\[ \begin{tikzcd} K_0(A) \arrow{r} & \underset{\longrightarrow}{\lim} \, (K_0(A), [H]) \arrow{r}{\cong} & K_0(\mathcal{O}_A(H) \rtimes_\gamma \mathbb{T}) \end{tikzcd} \]
Then we have $j \circ [H] = [\hat{\gamma}] \circ j$.

Suppose $\mathcal{O}_A(H)$ is AF\=/embeddable and $x \in K_0(A)$ with $[H](x) \leq x$.  Then
\[ [\hat{\gamma}](j(x)) = j([H](x)) \leq j(x) \]
and therefore Brown's Embedding Theorem implies
\[ j([H](x)) = [\hat{\gamma}](j(x)) = j(x). \]
It follows that $[H]^{n}(x - [H](x)) = 0$ for some sufficiently large integer $n$.  Since $[H]$ is faithful and $x - [H](x) \geq 0$, Proposition \ref{prop:FaithfulKTheory} implies $[H]x = x$.

Conversely, suppose condition (4) in the theorem holds and $x \in K_0(\mathcal{O}_A(H) \rtimes_\gamma \mathbb{T})$ such that $\hat{\gamma}(x) \leq x$.  Then there is a $y \in K_0(A)$ such that $j(y) = x$.  Since
\[ j([H](y)) = [\hat{\gamma}](j(y)) = [\hat{\gamma}](x) \leq x = j(y), \]
there is an integer $n$ such that $[H]^{n+1}(y) \leq [H]^n(y)$.  Set $z = [H]^n(y) \in K_0(A)$.  Then $j(z) = x$ and $[H](z) \leq z$.  Hence we have $[H](z) = z$ and
\[ [\hat{\gamma}](x) = [\hat{\gamma}](j(z)) = j([H](z)) = j(z) = x. \]
Thus Brown's Embedding Theorem implies $\mathcal{O}_A(H)$ is AF\=/embeddable.
\end{proof}

It was also shown in \cite{Schafhauser1} and \cite{Schafhauser2} that for many topological graph algebras (see \cite{KatsuraTGA} for a definition of topological graph algebras), stably finite implies AF\=/embeddable.  For instance, this holds for $C^*(E)$ where either $E$ is a discrete graph or $E$ is a compact graph with no sinks.  In light of Theorem C, this also holds when $E$ is a totally disconnected topological graph (or a topological quiver in the sense of \cite{MuhlyTomfordeQuivers}).

\begin{corollary}\label{cor:AFEmbeddingQuiver}
If $E$ is a second countable, totally disconnected, topological quiver, then the following are equivalent:
\begin{enumerate}
  \item $C^*(E)$ is AF\=/embeddable;
  \item $C^*(E)$ is quasidiagonal;
  \item $C^*(E)$ is stably finite.
\end{enumerate}
\end{corollary}

\begin{remark}
In the case $E$ is a discrete graph or $E$ is a compact topological graph with no sinks, the AF\=/embeddability of $C^*(E)$ can be characterized in terms of the structure of $E$.  It would be interesting to find a condition on a totally disconnected topological quiver $E$ which characterizes the AF\=/embeddability of $C^*(E)$; in fact this was part of the motivation for this projection.  It seems likely that if $E$ is a row--finite, totally disconnected, topological graph with no sources, then one relate the $K$\=/theoretic characterization of AF\=/embeddability in Theorem C to the combinatorial structure of the underlying graph $E$.
\end{remark}

Combining our main result with the work of Exel on partial crossed products in \cite{ExelCircleActions}, we obtain another application to AF\=/embeddability.  Given a $C^*$\=/algebra $A$ and a strongly continuous action $\gamma : \mathbb{T} \rightarrow \operatorname{Aut}(A)$, define
\[A_n := \{ a \in A : \gamma_z(a) = z^n a \text{ for every } z \in \mathbb{T} \}. \]
We say $\gamma$ is \emph{semi\-/saturated} if $A$ is generated as a $C^*$\=/algebra by $A_0$ and $A_1$.

\begin{corollary}\label{cor:CircleActions}
Suppose $A$ is a separable $C^*$\=/algebra and $\gamma: \mathbb{T} \rightarrow \operatorname{Aut}(A)$ is a strongly continuous, semi\-/saturated action on $A$.  If $A^\gamma$ is AF, then $A \rtimes_\gamma \mathbb{T}$ is AF and the following are equivalent:
\begin{enumerate}
  \item $A$ is AF--embeddable;
  \item $A$ is quasidiagonal;
  \item $A$ is stably finite.
\end{enumerate}
\end{corollary}

\begin{proof}
Suppose $\gamma : \mathbb{T} \rightarrow \operatorname{Aut}(A)$ is a strongly continuous, semi\-/saturated action such that $A^\gamma$ is AF.  Replacing $A$ with $A \otimes \mathbb{K}$ and $\gamma$ with $\gamma \otimes \operatorname{id}$, we may assume $\gamma$ is stable in the sense of \cite[Definition 4.21]{ExelCircleActions}.  Note that $\gamma$ is still semi\-/saturated and $A^\gamma$ is still AF.

By Theorem 4.21 of \cite{ExelCircleActions}, there are ideals $I$ and $J$ of $A^\gamma$ and an isomorphism $\theta : I \rightarrow J$ such that $A$ is isomorphic to the covariance algebra of the partial automorphism $(\theta, I, J)$.  Note that $I$ may be viewed as a right Hilbert $A^\gamma$\=/module with
\[ x \cdot a = \theta^{-1}(\theta(x)a) \qquad \text{and} \qquad \langle x, y \rangle = \theta(x^*y) \]
for $x, y \in I$ and $a \in A^\gamma$.  Moreover, $I$ is a $C^*$\=/correspondence over $A^\gamma$ with the left action $\lambda: A^\gamma \rightarrow \mathbb{B}(I)$ given by $\lambda(a)(x) = ax$ for $a \in A^\gamma$ and $x \in I$.

Katsura has shown that $\mathcal{O}_{A^\gamma}(I)$ is isomorphic to the covariance algebra of $(\theta, I, J)$ and hence is isomorphic to $A$ (see Section 3.6 of \cite{KatsuraCorr1}).  Since $A^\gamma$ is AF, the result follows from Theorem B and Theorem C.
\end{proof}

\begin{remark}
Given any strongly continuous action $\gamma: \mathbb{T} \rightarrow \operatorname{Aut}(A)$ on a $C^*$\=/algebra $A$, let $A^{(n)}$ denote the $C^*$\=/algebra generated by $A_0$ and $A_n$.  Then there is a strongly continuous, semi\-/saturated action $\gamma^{(n)}: \mathbb{T} \rightarrow \operatorname{Aut}(A^{(n)})$ given by $\gamma^{n}_z(a) = z^{1/n} \gamma(a)$ for $a \in A^{(n)}$ and $z \in \mathbb{T}$, where $z^{1/n}$ is \emph{any} $n$th root of $z$.  Moreover, $A^{(n)}_0 = A_0$ and the span of the $A^{(n)}$, $n \in \mathbb{Z}$, is dense in $A$ (see \cite[Proposition 2.5]{ExelCircleActions}).  It seems likely that one could use this fact to remove the semi\-/saturated assumption in Corollary \ref{cor:CircleActions}, however we were not able to prove this.
\end{remark}

\end{document}